\documentclass[12pt,letterpaper]{amsart}
\usepackage[pass]{geometry}
\usepackage{amssymb,amsmath,amsthm,amsgen}
\usepackage{tikz}
\usepackage{comment}
\usepackage{}
\usepackage{enumitem}

\newcommand{\old}[1]{}
\theoremstyle{plain}
\newtheorem{thm}{Theorem}[section]
\newtheorem{lem}[thm]{Lemma}
\newtheorem{conj}{Conjecture}
\newtheorem{cor}[thm]{Corollary}

\newtheorem{prop}[thm]{Proposition}
\theoremstyle{definition}
\newtheorem{defn}[thm]{Definition}

\newtheorem{ex}[thm]{Example}
\newtheorem{rk}[thm]{Remark}
\newtheorem{qn}[thm]{Question}

\numberwithin{equation}{section}

\def\la{{{\lambda}}}

 at 10truept

\title{Variations on the $S_n$-module $Lie_n$}

\author{Sheila Sundaram}
\address{Pierrepont School, One Sylvan Road North, Westport, CT 06880}
\email{shsund@comcast.net}
\date{\today}

\thanks{(Combined Version)}


\begin{document}
\begin{abstract} We define, for each subset $S$ of primes, an $S_n$-module $Lie_n^S$ 
with interesting properties. When $S=\emptyset,$ this is the well-known representation $Lie_n$ of $S_n$ afforded by the free Lie algebra.

The most intriguing case is $S=\{2\},$  giving a decomposition of the regular representation as a sum of {exterior} powers of modules $Lie_n^{(2)}.$ This is in contrast to the theorems of Poincar\'e-Birkhoff-Witt and Thrall  which decompose the regular representation into a sum of symmetrised $Lie$ modules. We show that nearly every known property of $Lie_n$ has a counterpart for the module $Lie_n^{(2)},$ suggesting connections to the cohomology of configuration spaces via the character formulas of Sundaram and Welker, to the Eulerian idempotents of Gerstenhaber and Schack, and to the Hodge decomposition of the complex of injective words arising from  Hochschild homology, due to  Hanlon and Hersh. 

For arbitrary $S,$ the symmetric and exterior powers of the module $Lie_n^S$ allow us to deduce Schur positivity for a new class of multiplicity-free sums of power sums.

\end{abstract}
\maketitle


\section{Introduction}
In this paper we present 
 the unexpected discovery of a curious variant of  the $S_n$-module    
 $Lie_n$ afforded by the multilinear component of the free Lie algebra with $n$ generators. The theorems of Poincar\'e-Birkhoff-Witt  and Thrall (see, e.g. \cite{R}) state that the universal enveloping algebra of the free Lie algebra is the symmetric algebra over the free Lie algebra, and hence coincides with the full tensor algebra. This is equivalent, via Schur-Weyl duality, to Thrall's decomposition of the regular representation into a sum of symmetric powers of the representations $Lie_n.$ By contrast, here we obtain a decomposition of the regular representation as a sum of exterior powers of modules (Theorem \ref{ExtReg}).  The key ingredient is our variant of $Lie_n,$ an $S_n$-module that we denote by $Lie_n^{(2)},$ which turns out to possess remarkable properties akin to those of  $Lie_n.$  Our results (see Theorems \ref{Compare},  
 \ref{LehrerHLLieSup2}, \ref{EquivLieSup2}) 
 bear a striking resemblance to properties of the Whitney homology of the partition lattice (and hence the Orlik-Solomon algebra for the root system $A_n$),  and to the computation of the cohomology of  the configuration space for the braid arrangement found in \cite[Theorem 4.4]{SW}.
 In particular these properties indicate the possibility of an 
underlying algebra structure for $Lie_n^{(2)}$ involving an acyclic complex. Theorem \ref{EquivLieSup2} furthers this analogy;  we show that $Lie_n^{(2)}$ admits a filtration close to the one arising from the derived series of the free Lie algebra.  There is an interesting  action on derangements  arising from $Lie_n^{(2)}$ as well (Theorem \ref{HodgeFilt}); we prove that $Lie_n^{(2)}$ gives rise to a new decomposition of the homology of the complex of injective words studied by Reiner and Webb \cite{RW}, one that is different from the Hodge decomposition of Hanlon and Hersh \cite{HH}. These results are collected in Section 2,  showing that for every well-known property of  $Lie_n,$  the representation $Lie_n^{(2)}$  offers an interesting counterpart.

A characteristic feature of the  complement of the $A_{n-1}$-hyperplane arrangement in complex space, and hence the configuration spaces associated to the braid arrangement, is that the cohomology ring has the structure of a symmetric or exterior algebra over the top  cohomology as an $S_n$-module  (see Theorem \ref{SWThm4.4} and Theorem \ref{Compare}).   Moreover this top cohomology is  $Lie_n$ or its sign-tensored version, and thus its character values are supported on a specific class of permutations: those whose cycles all have the same length.  This is precisely the framework of the symmetric function identities  developed in \cite{Su1}; consequently these identities constitute the crucial methods of this paper, described in Section 5.  Here we  state the \lq\lq meta theorem\rq\rq from \cite{Su1} and develop  further methods that are directly applicable to studying  modules induced from centralisers.  

The module $Lie_n^{(2)}$ is a special case of a family of variations of $Lie_n,$ whose discovery arises from the original motivation for this paper, namely the investigation begun in \cite{Su1} on the positivity of the row sums in the character table of $S_n.$  
For each irreducible character $\chi^\lambda$ 
and any subset $T$ of the conjugacy classes, one can form the sum $\sum_{\mu\in T} \chi^\lambda(\mu),$ and ask when this sum is nonnegative.   
A  method for generating such  classes of subsets $T$ was presented in \cite{Su1}.  This paper  generalises and extends these results with classes indexed by subsets of primes (Theorems 3.6 and 3.9).  The approach in Section 3 is to define certain $S_n$-modules $Lie_n^S$ induced from the cyclic subgroup  generated by an $n$-cycle, indexed by subsets $S$ of primes, and show that their symmetric powers 
can be expressed as the product $\prod_{n\in P( S)} (1-p_n)^{-1},$  for an appropriate subset $P(S)$ determined by $S.$  The interesting choices of $S$ other than $Lie^{(2)}$ are recorded in Theorem \ref{OnePrime} and Theorems \ref{Spositivity1} and \ref{Spositivity2}.  These results  were announced in \cite{SuFPSAC2018}.

 A second approach to the Schur positivity problem for power sums is developed in Section 4.  We give a general formula which expresses the product $\prod_{n\in T} (1-p_n)^{-1}$ as a symmetrised module over a sequence of  possibly virtual representations $f_n^T,$ having the specific property that their characters vanish unless the conjugacy class has all cycles of equal length.  The goal then is to determine for what choices of the set $T$ the $f_n^T$ are true 
$S_n$-modules, thereby establishing the Schur positivity of the 
product $\prod_{n\in T} (1-p_n)^{-1}.$    The module $Lie_n$ plays a prominent role in the construction.
In the course of these calculations many interesting plethystic identities emerge, as well as many  new conjectures on Schur positivity.  The results of this section were announced in \cite{SuFPSAC2019}.

The paper \cite{Su1} develops machinery which greatly facilitates and unifies the plethystic computations involved in finding symmetric and exterior powers, or equivalently, the higher Lie analogues, of a certain class of $S_n$-modules. Section~\ref{Secmetathms} summarises these important tools, which are key to the results of this paper.  We formulate and unify, in a general setting, many of the properties that have been discovered in the literature for various representations.  In particular this section  may be viewed as a toolkit for dealing with the type of representations that seem to arise in the cohomology of configuration spaces.  One important consequence of the results of this section is a fact that does not appear to have been previousy observed, namely the \textit{equivalence} of all  the known representation-theoretic properties of $Lie$ (the formulas of Thrall and Cadogan, the filtration arising from the derived series, the appearance of the $Lie$ character in the action on derangements).   This is explained in Theorem \ref{EquivPBW}.  Finally Section 6 contains a miscellany of useful plethystic facts, and identities that have frequently come up in different homological contexts.

\subsection{Preliminaries}
We follow \cite{M} and \cite{St4EC2} for notation regarding symmetric functions.  In particular, $h_n,$ $e_n$ and $p_n$ denote respectively the complete homogeneous, elementary and power-sum symmetric functions.  If $\mathrm{ch}$ is the Frobenius characteristic map from the representation ring of the symmetric group $S_n$ to the ring of symmetric functions with real coefficients, then 
$h_n=\mathrm{ch}(1_{S_n})$ is the characteristic of the trivial representation, and $e_n=\mathrm{ch}({\rm sgn}_{S_n})$ is the characteristic of the sign representation of $S_n.$   If $\mu$ is a partition of $n$ then 
define $p_\mu =\prod_i p_{\mu_i};$ $h_\mu$ and $e_\mu$ are defined multiplicatively in analogous fashion.  As in \cite{M}, 
the Schur function $s_\mu$  indexed by the partition $\mu$ is the Frobenius characteristic of the $S_n$-irreducible indexed by $\mu.$  
  Finally, $\omega$ is the involution on the ring of symmetric functions which takes $h_n$ to $e_n,$ corresponding to tensoring with the sign representation.

If $q$ and $r$ are characteristics of representations of $S_m$ and $S_n$ respectively, they yield a representation of the wreath product $S_m[S_n]$ in a natural way, with the property that when this representation is induced up to $S_{mn},$ its Frobenius characteristic is the plethysm $q[r].$ For more background about this operation, see \cite{M}.  We will make extensive use of the properties of this operation, in particular the fact that plethysm with a symmetric function $r$ is an endomorphism on the ring of symmetric functions \cite[(8.3)]{M}.  See also \cite[Chapter 7, Appendix 2, A2.6]{St4EC2}.
Define \begin{align}\label{defHE} &H(t)=\sum_{i\geq 0}t^i  h_i, \quad E(t) = \sum_{i\geq 0} t^i  e_i; \\
&H=\sum_{i\geq 0}  h_i, \quad E= \sum_{i\geq 0}  e_i; \quad H^{\pm}=\sum_{r\geq 0} (-1)^r h_r, \quad E^{\pm}=\sum_{r\geq 0} (-1)^r e_r. 
\end{align} 
Now let $\{q_i\}_{i\geq 1}$ be a sequence of symmetric functions, each $q_i$ homogeneous of degree $i.$  
Let $Q=\sum_{i\geq 1}q_i,$  $Q(t)=\sum_{n\geq 1} t^n q_n.$
For each partition $\lambda$ of $n\geq 1$  with $m_i(\lambda)=m_i$ parts equal to $i\geq 1,$ let $|\lambda|=n = \sum_{i\geq 1} i m_i$ be the size of $\lambda,$ and 
$\ell(\lambda)=\sum_{i\geq 1} m_i(\lambda)=\sum_{i\geq 1} m_i$ be the length (total number of parts) of $\lambda.$ 

Define 
\begin{equation}\label{HigherQ}
 H_\lambda[Q]=\prod_{i:m_i(\lambda)\geq 1} h_{m_i}[q_i],\qquad  \qquad E_\lambda[Q]=\prod_{i:m_i(\lambda)\geq 1} e_{m_i}[q_i].
 \end{equation}

For the empty partition (of zero) we define $H_\emptyset [Q]=1=
E_\emptyset[Q]=H^{\pm}_\emptyset [Q]=
E^{\pm}_\emptyset[Q].$

%

 Consider the generating functions 
$H[Q](t)$ and $E[Q](t).$ 
With the convention that $Par$, the set of all partitions of nonnegative integers,
 includes the unique empty partition of zero,  by the preceding observations and standard properties of plethysm \cite{M} we have 
 
 \begin{equation}\label{defhrofQ}h_r[Q]|_{{\rm deg\ }n}=\sum_{\stackrel {\lambda\vdash n}{\ell(\lambda)=r}} H_\lambda[Q], 
 \qquad \text{and } \quad 
 e_r[Q]|_{{\rm deg\ }n}=\sum_{\stackrel{\lambda\vdash n}{\ell(\lambda)=r}}E_\lambda[Q];
 \end{equation}

$$H[Q](t)=\sum_{\lambda\in Par} t^{|\lambda|} H_\lambda[Q], \qquad \text{and } \quad E[Q](t)=\sum_{\lambda\in Par} t^{|\lambda|} E_\lambda[Q].$$  
We will also let $D\!Par$ denote the subset of $Par$ consisting of the empty partition and all partitions with distinct parts.

Also write $Q^{alt}(t)$ for the alternating sum $\sum_{n\geq 1} (-1)^{i-1} t^i q_i = t q_1-t^2 q_2+t^3 q_3-\ldots.$  

Let $\psi(n)$ be any real-valued function defined on the positive integers. 
Define symmetric functions $f_n$ by 
 \begin{equation}\label{definef_n}f_n = \dfrac{1}{n} \sum_{d|n} \psi(d) p_d^{\frac{n}{d}},
\quad \text{so that} \quad 
\omega(f_n) =  \dfrac{1}{n} \sum_{d|n} \psi(d) (-1)^{n-\frac{n}{d}} p_d^{\frac{n}{d}}.\end{equation}
Note that,  when $\psi(1)$ is a positive integer,  this makes $f_n$ the Frobenius characteristic of a possibly virtual $S_n$-module whose dimension is $(n-1)!\psi(1).$

Also define the associated polynomial in one variable, $t,$ by
\begin{equation}\label{definepolyf_n}
f_n(t) =\dfrac{1}{n} \sum_{d|n} \psi(d) t^{\frac{n}{d}}.
\end{equation}


\section{A comparison of $Lie_n$ and the variant $Lie_n^{(2)}$}

The results of this section and the next were announced in \cite{SuFPSAC2018}.

In this section we define  the $S_n$-module $Lie_n^{(2)}$ and describe some of its remarkable properties.  The goal here is to highlight and analyse this module in a representation-theoretic and homological context, and to explain how to interpret the plethystic identities.  The  properties are established using symmetric function techniques applied to the Frobenius characteristic of $Lie_n^{(2)}$;  we have relegated the technical details of the proofs to end of the next section, following Theorem \ref{PlInv-E}.   More details appear in the last two sections of the paper.  The present section has been written to be self-contained.

Recall \cite{R} that the $S_n$-module $Lie_n$ is the action of $S_n$ on the multilinear component of the free Lie algebra, and coincides with the  induced representation $\exp(\frac{2i\pi}{n})\uparrow_{C_n}^{S_n},$ where $C_n$ is the cyclic group generated by an $n$-cycle in $S_n$.  Its Frobenius characteristic is obtained by taking $\psi(d)=\mu(d)$ (the number-theoretic M\"obius function) in equation~(\ref{definef_n}).

Another module that will be of interest is the $S_n$-module $Conj_n$ afforded by the conjugacy action of $S_n$ on the class of $n$-cycles.  Clearly we have $Conj_n\simeq {\mathbf 1}\uparrow_{C_n}^{S_n}.$  Its Frobenius characteristic is obtained by taking $\psi(d)=\phi(d)$ (Euler's totient function) in equation~(\ref{definef_n}).


\begin{defn}\label{defLieSup2} Let $k_n$ be the highest power of 2 dividing $n.$ Define $Lie_n^{(2)}$ to be the induced module 
$$\exp(\frac{2i\pi}{n}\cdot 2^{k_n})\uparrow_{C_n}^{S_n}.$$
\end{defn}

The first two of the following facts are immediate.  $Lie_n^{(2)}$ is $S_n$-isomorphic to 
\begin{itemize} 
\item   $ Lie_n$ if $n$ is odd;
\item $Conj_n$ if $n$ is a power of 2;
\item $Lie_n\otimes {\bf sgn}_{S_n}$ if $n$ is twice an odd number.
\end{itemize}

The third fact follows, for example, by first establishing the isomorphism 
$${\bf sgn}_{S_n}\otimes \chi\uparrow_{C_n}^{S_n}\simeq ({\bf sgn}^{n-1}_{C_n}\otimes \chi)\uparrow_{C_n}^{S_n}.$$
(A different proof using Frobenius characteristics is also possible.)

   In order to elaborate on the properties of $Lie_n^{(2)}$ described in the Introduction, it is most convenient to use the language of symmetric functions.  Thus we will often abuse notation and use $Lie_n$ and $Lie_n^{(2)}$ to mean both the module and its Frobenius characteristic.
   
 We write $Lie$ for the sum of symmetric functions $\sum_{n\geq 1} Lie_n$ and $Lie^{(2)}$ for the sum 
of symmetric functions $\sum_{n\geq 1} Lie_n^{(2)}.$
Recall from Section 1.1 that we define,  for each partition $\lambda$ of $n\geq 1$  with $m_i(\lambda)=m_i$ parts equal to $i,$
 $H_\lambda[Q]=\prod_{i:m_i(\lambda)\geq 1} h_{m_i}[q_i]$ and $E_\lambda[Q]=\prod_{i:m_i(\lambda)\geq 1} e_{m_i}[q_i];$  see also equation (\ref{defhrofQ}).
 Finally, recall that $p_1^n=h_1^n=e_1^n$ is the Frobenius characteristic of the regular representation ${\mathbf 1}\uparrow_{S_1}^{S_n}$ of $S_n.$
 
 The $H_\lambda[Lie]$ are the (Frobenius characteristics of) the \textit{higher Lie} modules appearing in Thrall's decomposition of the regular representation (see below and later sections for more details).  
 We denote the wreath product of $S_a$ with $a$ copies of $S_b$ by $S_a[S_b];$ explicitly it is the normaliser of the direct product 
 $\underbrace{S_b\times \ldots \times S_b}_{a}$  in $S_{ab}.$ 
 If $V_a$ and $V_b$ are respectively representations of $S_a, S_b,$ 
there is an obvious associated representation $V_a[V_b]$ of the wreath product $S_a[S_b]$, whose Frobenius characteristic is given by the plethysm ${\rm ch\,}V_a[{\rm ch\,}V_b].$  The higher Lie module 
$H_\lambda[Lie],$ for a partition $\lambda$ of $n$ with 
$m_i$ parts equal to $i,$  is the characteristic of  the induced representation 
\begin{equation*}{\Large\otimes}_i {\bf 1}_{S_{m_i}}[Lie_i]
{\large\uparrow}_{\prod_i S_{m_i}[S_i]}^{S_n}.\end{equation*}

If $X$ is any topological space, then the ordered configuration space $Con\! f_n \, X$ of $n$ distinct points in $X$ is defined to be the set 
$\{(x_1,\ldots, x_n): i\neq j \Longrightarrow x_i\neq x_j\}.$ The symmetric group $S_n$ acts on $Con\! f_n\, X$ by permuting coordinates, and hence induces an action on the cohomology $H^k(Con\! f_n\,  X, {\mathbb Q}), k\geq 0.$  

\begin{thm}\label{SWThm4.4}\cite[Theorem 4.4, Corollary 4.5]{SW} For all $d\geq 1, $  and $0\leq k\leq n-1,$ the Frobenius characteristic of 
\begin{itemize}
\item 
$H^k( Con\! f_n\, {\mathbb R}^{2},{\mathbb Q} ) \simeq H^{(2d-1)k}(Con\! f_n\, {\mathbb R}^{2d},{\mathbb Q})$ is   
$\omega\left(e_{n-k}[Lie]|_{\text{deg }n} \right).$
\item
$H^{2k}( Con\! f_n\, {\mathbb R}^{3} ,{\mathbb Q}  ) \simeq H^{2dk}(Con\! f_n\, {\mathbb R}^{2d+1}, {\mathbb Q})$   is 
$
h_{n-k}[Lie]|_{\text{deg }n} .$
\end{itemize}
The cohomology vanishes in all other degrees.

When $d=1,$ $H^{0}( Con\! f_n\, {\mathbb R} ,{\mathbb Q}  )$ carries the regular representation of $S_n.$
\end{thm}

We will use the cohomology of $ Con\! f_n\, {\mathbb R}^{2} $ as the prototype for the configuration spaces of even-dimensional Euclidean space, and  $ Con\! f_n\, {\mathbb R}^{3} $ as the prototype for 
the configuration spaces of odd-dimensional Euclidean space.  Note that cohomology is concentrated in all degrees in the former (more generally in all multiples of $(2d-1)$ for $2d$-dimensional space), and only in even degrees in the latter.

The results of this section will show that the representation $Lie_n^{(2)}$ has properties  curiously parallelling those of $Lie_n.$ 
Theorem \ref{SWThm4.4} above states the ``Lie" identities of
Theorem \ref{Compare} below in the context of the configuration spaces of $X={\mathbb R}^2$ and $X={\mathbb R}^3.$   The module $Lie_n$ arises as the highest nonvanishing cohomology for the configuration space of ${\mathbb R}^d, d$ odd, and when tensored with the sign, as the highest nonvanishing cohomology for the configuration space of ${\mathbb R}^d, d$ even.  This is the classically known 
prototype;  the variant $Lie_n^{(2)}$ will be shown to closely follow its example.

\begin{thm}\label{Compare}  The  symmetric function $Lie_n^{(2)}$ satisfies the following plethystic identities, analogous to $Lie_n$.  

\begin{equation}\label{SymExt}
\sum_{\lambda\vdash n}H_\lambda[Lie]=p_1^n; \qquad\qquad \sum_{\lambda\vdash n}E_\lambda[Lie^{(2)}]=p_1^n;
\end{equation}
\begin{equation}\label{PlInvHE}
H[\sum_{n\geq 1} (-1)^{n-1}\omega(Lie_n)]=1+p_1;
\qquad E[\sum_{n\geq 1} (-1)^{n-1}\omega(Lie^{(2)}_n)]=1+p_1
\end{equation}
\begin{equation}\label{AcycEH}
\text{If } n\geq 2, \ \sum_{\lambda\vdash n}(-1)^{n-\ell(\lambda)}E_\lambda[Lie]=0; \qquad\qquad \sum_{\lambda\vdash n}(-1)^{n-\ell(\lambda)}H_\lambda[Lie^{(2)}]=0;
\end{equation}
%
\begin{equation}\label{TotalCoh}
\text{If } n\geq 2, \ \sum_{\lambda\vdash n}E_\lambda[Lie]=2e_2 p_1^{n-2}; \qquad\qquad \sum_{\lambda\vdash n}H_\lambda[Lie^{(2)}]=\sum_{\lambda\vdash n, \lambda_i=2^{a_i}} p_\lambda.
\end{equation}
 Moreover,  the $Lie$ identities are all equivalent, and the $Lie^{(2)}$ identities are also equivalent.
\end{thm}

We now discuss the implications of Theorem \ref{Compare}. 

\vspace{.07in}
\noindent {\bf \textit{Equation (\ref{SymExt})}:}

  The first equation in (\ref{SymExt}) is simply Thrall's classical theorem \cite{T}, rederived in Theorem \ref{ThrallPBWCadoganSolomon}, stating that the regular representation of $S_n$ decomposes into a sum of symmetrised modules induced from the centralisers of $S_n,$ the Lie modules.  Thrall's theorem in this context is equivalent to the Poincar\'e-Birkhoff-Witt Theorem, which states that the universal enveloping algebra of the free Lie algebra is its symmetric algebra \cite{R}.  Recall that the Lefschetz module of a complex is the alternating sum by degree of the homology modules. In view of Theorem \ref{SWThm4.4}, since cohomology is nonzero only in even degrees, the Lefschetz module is in fact a sum of homology modules, and this can in turn be reinterpreted as saying that:
  \begin{prop}\label{LefschetzR3}
   The regular representation of $S_n$  is carried by the Lefschetz module of  $Con\!f_n\, {\mathbb R}^3,$ and more generally $Con\!f_n\, {\mathbb R}^d$ for odd $d$, which coincides with its cohomology ring and  is  isomorphic to the symmetric algebra over the top cohomology.
   \end{prop}

  The second equation in (\ref{SymExt}) is our new result.  It gives a new decomposition of the regular representation: 
  
  \begin{thm}\label{ExtReg}   The regular representation decomposes into  a sum of \textit{exterior} powers of modules induced from the centralisers of $S_n,$ namely the modules $Lie_n^{(2)}$.
  \end{thm}

\noindent {\bf \textit{Equation (\ref{PlInvHE})}:}

In (\ref{PlInvHE}), the second equation is new, giving the plethystic inverse of the elementary symmetric functions $\sum_{n\geq 1} e_n,$  while 
the first equation contains the known result of Cadogan \cite{C} (see Theorem \ref{ThrallPBWCadoganSolomon}) giving the plethystic inverse of the homogeneous symmetric functions $\sum_{n\geq 1} h_n$.

\vspace{.07in}
\noindent {\bf \textit{Equation (\ref{AcycEH})}:}

The equations in  (\ref{AcycEH}) and (\ref{TotalCoh}) are particularly significant.  It is well known that the degree $n$ term in the plethysm $e_{n-r}[Lie]$ is the Frobenius characteristic of the $r$th-Whitney homology $W\!H_r(\Pi_n)$ of the partition lattice $\Pi_n,$ tensored with the sign (see \cite[Remark 1.8.1]{Su0}),  and hence of the sign-tensored $r$th cohomology $H^r( Con\! f_n\,{\mathbb R}^{2} )$ of Theorem \ref{SWThm4.4}. The $r$th Whitney homology also coincides as an $S_n$-module with the $r$th cohomology of the pure braid group, see \cite{HL}.
The first equation in \ref{AcycEH} therefore restates the acyclicity of Whitney homology for the partition lattice \cite{Su0}, and hence also says  (in contrast to the odd case $ Con\! f_n\,{\mathbb R}^{3} $ of Proposition~\ref{LefschetzR3} above) that :
\begin{prop}\label{LefschetzR2}
The Lefschetz module of $Con\! f_n\,{\mathbb R}^{2}$ (and more generally $Con\! f_n\,{\mathbb R}^{2d}$ for even $d$) vanishes identically.
\end{prop}
Writing $W\!H_{odd}(\Pi_n)$ for $\oplus_{k=0}^{n/2} W\!H_{2k+1}(\Pi_n),$ and $W\!H_{even}(\Pi_n)$ for $\oplus_{k=0}^{n/2} W\!H_{2k}(\Pi_n),$ we have the isomorphism of $S_n$-modules
\begin{equation} \label{EvenOdd}W\!H_{odd}(\Pi_n)\simeq W\!H_{even}(\Pi_n), \quad n\geq 2.\end{equation}
\noindent {\bf \textit{Equation (\ref{TotalCoh})}:}

  Denote by $W\!H(\Pi_n)$ the sum of all the graded pieces of the Whitney homology of $\Pi_n.$ 
The first equation in (\ref{TotalCoh}) says (recall that we have tensored with the sign representation) that 
 $W\!H(\Pi_n)=2\, ({\mathbf 1}\uparrow_{S_2}^{S_n}), \quad n\geq 2,$ 
 a result originally due to Lehrer, who proved that this is the $S_n$-representation on the cohomology ring $H^*( Con\! f_n\,{\mathbb R}^{2})$  (Lehrer actually considers the cohomology of the complement of the braid arrangement of type $A_{n-1}$ \cite[Proposition 5.6 (i)]{Le}).  We may rewrite this in our notation as 
\begin{equation}\label{Lehrer-a}   H^*( Con\! f_n\,{\mathbb R}^{2} )=W\!H(\Pi_n)={\rm ch}^{-1}(2 h_2 p_1^{n-2})=2\ ({\mathbf 1}\uparrow_{S_2}^{S_n}) , \quad n\geq 2.\end{equation}

Note that the first equation in (\ref{TotalCoh}) also confirms the following theorem of Orlik and Solomon.
\begin{prop}\label{Orlik-Solomon} \cite{LS}
$H^*( Con\! f_n\,{\mathbb R}^{2}) $ 
has the structure of an exterior algebra over the top cohomology.
\end{prop}

By combining equation (\ref{Lehrer-a}) with (\ref{EvenOdd}), we obtain
\begin{equation}\label{Lehrer-b} H^{odd}(Con\! f_n\,{\mathbb R}^{2} )
\simeq H^{even}(Con\! f_n\,{\mathbb R}^{2})\simeq {\mathbf 1}\uparrow_{S_2}^{S_n}, \quad n\geq 2.
\end{equation}
yielding the decomposition of the regular representation noticed by   Hyde and Lagarias \cite{HL}:
\begin{equation}\label{HL} 
H^{odd}( Con\! f_n\,{\mathbb R}^{2})\oplus 
\textbf{ sgn}\otimes H^{even}(Con\! f_n\,{\mathbb R}^{2}) \simeq {\mathbf 1}\uparrow_{S_1}^{S_n}. \end{equation}


From \ref{Lehrer-a} it also follows that 
\begin{equation}\label{IndConf}
H^*( Con\! f_{n+1}\,{\mathbb R}^{2})\simeq H^*(Con\! f_n\, {\mathbb R}^{2}) \uparrow_{S_{n}}^{S_{n+1}}.
\end{equation}



We now describe results of a similar flavour for the new  representation $Lie_n^{(2)}.$ 
Define a new module $V\!h_r(n)$ whose Frobenius characteristic is the degree $n$ term in $h_{n-r}[Lie^{(2)}];$ this is a true $S_n$-module. The second equation of (\ref{AcycEH}) can now be interpreted as an acylicity statement:
$$V\!h_n(n)-V\!h_{n-1}(n) + V\!h_{n-2}(n)-\ldots +(-1)^r V\!h_r(n) +\ldots=0, \quad n\geq 2.$$
and hence, in analogy with (\ref{EvenOdd}), letting $V\!h_{odd}(n)=\oplus_{k=0}^{n/2} V\!h_{2k+1}$ and $V\!h_{even}=\oplus_{k=0}^{n/2} V\!h_{2k}:$
\begin{equation}\label{EvenOddLieSup2} V\!h_{odd}(n)\simeq V\!h_{even}(n), \quad n\geq 2.\end{equation}

The second equation in (\ref{TotalCoh}) gives, similarly, 
\begin{equation} {\rm ch\,}(V\!h_{odd}(n)\oplus V\!h_{even}(n))=
\sum_{\lambda\vdash n; \lambda_i=2^{a_i}} p_\lambda
\end{equation}
Hence we have established the following results, analogous to (\ref{Lehrer-a})-(\ref{IndConf}):

\begin{thm}\label{LehrerHLLieSup2}  The following $S_n$-equivariant isomorphisms hold for the modules $Vh_r(n)={\rm ch}^{-1}\, h_{n-r}[Lie^{(2)}]\vert_{{\rm deg\ }n}$, giving   Schur-positive  functions with integer coefficients.
\begin{equation}\label{LehrerSup2b}V\!h_{odd}(n)\simeq V\!h_{even}(n)={\rm ch}^{-1}\,
\frac{1}{2}\sum_{\lambda\vdash n; \lambda_i=2^{a_i}} p_\lambda
\end{equation}
\vskip-.2in
\begin{equation}\label{LehrerSup2a} V\!h(n)=V\!h_{odd}(n)\oplus V\!h_{even}(n)={\rm ch}^{-1}\,
\sum_{\lambda\vdash n; \lambda_i=2^{a_i}} p_\lambda
\end{equation}
\vskip-.2in
\begin{equation}\label{HLSup2}
V\!h_{odd}(n)\oplus {\,\bf sgn}_{S_n}\otimes V\!h_{odd}(n)={\rm ch}^{-1}\,
\sum_{\lambda\vdash n; n-\ell(\lambda) {\text even };\lambda_i=2^{a_i}} p_\lambda
\end{equation}
\vskip-.2in
\begin{equation}\label{IndSup2}
V\!h(2n+1)\simeq V\!h(2n)\uparrow_{S_{2n}}^{S_{2n+1}}.
\end{equation}
\end{thm}
We now have at least four decompositions of the regular representation, namely the two in (\ref{SymExt}) and two from (\ref{HL}) (tensoring the latter with the sign representation gives two),  into sums of modules indexed by the conjugacy classes,  each module  obtained by  inducing a linear character from a centraliser of $S_n.$  We write these out  for $S_4$ and $S_5$ to show that they are indeed all distinct.    In the two tables below, each column adds up to the regular representation. Note that  $Lie_4^{(2)}$ coincides with $Conj_4,$ while $Lie_5^{(2)}$ is just $Lie_5.$ Hence these modules appear in the last row of each table.
%
\begin{ex} The first two decompositions are from equation (\ref{SymExt})  of Theorem \ref{Compare}; the third is from equation (\ref{HL}).
In all cases, of course,  the four pieces for $S_4$ (respectively, the five pieces for $S_5$)  each have the same dimension, equal to the sum of the sizes of the constituent conjugacy classes, namely,  $1,6,11,6$ (respectively $1,10, 35, 50, 24$).  Note that the conjugacy classes are grouped together by number of disjoint cycles, i.e. by length $\ell$ of the corresponding partition.  That these four decompositions are all distinct is clear, since each has a distinguishing feature.  E.g. for  $S_4,$ both copies of the irreducible for the partition $(2^2)$ appear only in one graded piece for {\bf [PBW]}, while the reflection representation is a submodule of one graded piece only in the third.  
\vskip.05in
\begin{center} 
{\small Table 1: The regular representation of $S_4$}\\  \nopagebreak
{\small (Poincar\'e polynomial $1+6t+11t^2+6t^3$)}\end{center}
\begin{center}
\begin{tabular}{|c|c|c|c|}\hline
{\small Conjugacy}  &  {\small\bf PBW} ($Con\!f\,{\mathbb R}^3$)
    &  {\small\bf Ext} & {\small\bf Whitney} ($Con\!f\,{\mathbb R}^2$) \\
    {\small classes} &{\small irreducibles} &{\small irreducibles} &{\small irreducibles } \\ \hline
    ${\scriptstyle (1^4)}$ & ${\scriptstyle h_4[Lie]|_{\text{deg }4}}$ 
    & ${\scriptstyle e_4[Lie^{(2)}]|_{\text{deg }4}}$
     & ${\scriptstyle\omega(W\!H_0)}$\\
 ${\scriptstyle \ell=4}$   &${\scriptstyle (4)}$ &${\scriptstyle (1^4)}$ &${\scriptstyle (1^4)}$\\ \hline
    ${\scriptstyle (2,1^2)}$ & ${\scriptstyle h_3[Lie]|_{\text{deg }4}}$ & ${\scriptstyle e_3[Lie^{(2)}]|_{\text{deg }4}}$ 
    & ${\scriptstyle W\!H_1}$\\
  ${\scriptstyle \ell=3}$  &${\scriptstyle (3,1)+(2,1^2)}$ & ${\scriptstyle (3,1)+(2,1^2)}$ & ${\scriptstyle (4)+(3,1)+(2^2)}$\\ \hline
    ${\scriptstyle (3,1)\text{ and }(2^2)}$ &${\scriptstyle h_2[Lie]|_{\text{deg }4}}$
    &${\scriptstyle e_2[Lie^{(2)}]|_{\text{deg }4}}$
    &${\scriptstyle\omega(W\!H_2)}$\\
 ${\scriptstyle \ell=2}$   & ${\scriptstyle (3,1)+2(2^2)+(2,1^2) +(1^4)}$
    &${\scriptstyle 2(3,1)+(2^2)+(2,1^2)}$ 
    &${\scriptstyle (3,1)+2 (2,1^2)+(2^2)}$\\ \hline
    ${\scriptstyle (4)}$ & ${\scriptstyle h_1[Lie]|_{\text{deg }4}}$ 
    &${\scriptstyle e_1[Lie^{(2)}]|_{\text{deg }4}}$
    &${\scriptstyle W\!H_3=\omega(Lie_4)}$ \\ 
 ${\scriptstyle \ell=1}$   & ${\scriptstyle (3,1)+(2,1^2)}$ &${\scriptstyle (4)+(2^2)+(2,1^2)} $ &${\scriptstyle (3,1)+(2,1^2)}$\\ \hline
\end{tabular}
\end{center}
%
\vfill\eject
%
\begin{center} {\small Table 2: The regular representation of $S_5$ }\\ {\small (Poincar\'e polynomial $1+10t+35t^2+50t^3+24t^4 $)}\end{center}
\nopagebreak
\begin{center}
\begin{tabular}{|c|c|c|c|}\hline
{\small Conjugacy}  &  {\small\bf PBW}($Con\!f\,{\mathbb R}^3$)
    &  {\small\bf Ext} & {\small\bf Whitney}($Con\!f\,{\mathbb R}^2$) \\
   {\small  classes} &{\small irreducibles} &{\small irreducibles} &{\small irreducibles } \\ \hline
    ${\scriptstyle (1^5)}$ & ${\scriptstyle h_5[Lie]|_{\text{deg }5}}$ 
    & ${\scriptstyle e_5[Lie^{(2)}]|_{\text{deg }5}}$
     & ${\scriptstyle\omega(W\!H_0)}$\\
 ${\scriptstyle \ell=5}$   &${\scriptstyle (5)}$ &${\scriptstyle (1^5)}$ &${\scriptstyle (1^5)}$\\ \hline
    ${\scriptstyle (2,1^3)}$ & ${\scriptstyle h_4[Lie]|_{\text{deg }5}}$ & ${\scriptstyle e_4[Lie^{(2)}]|_{\text{deg }5}}$ 
    & ${\scriptstyle W\!H_1}$\\
${\scriptstyle \ell=4}$    &${\scriptstyle (4,1)+(3,1^2)}$ & ${\scriptstyle (3,1^2)+(2,1^3)}$ & ${\scriptstyle (5)+(4,1)+(3,2)}$\\ \hline
    ${\scriptstyle (3,1^2) \text{ and }(2^2,1)}$ &${\scriptstyle h_3[Lie]|_{\text{deg }5}}$
    &${\scriptstyle e_3[Lie^{(2)}]|_{\text{deg }5}}$
    &${\scriptstyle\omega(W\!H_2)}$\\
 ${\scriptstyle \ell=3}$   & ${\scriptstyle (4,1)+2(3,2)+(3,1^2)}$ 
    &${\scriptstyle (4,1)+2(3,2)+2(3,1^2)}$
    &${\scriptstyle (3,2)+2(3,1^2)}$\\ 
&${\scriptstyle +2(2^2,1)+(2,1^3)+(1^5)}$ &${\scriptstyle +(2^2,1)+(2,1^3)}$ & ${\scriptstyle +2(2^2,1)+2(2,1^3)
}$\\
\hline
      ${\scriptstyle (4,1)\text{ and }(3,2)}$ &${\scriptstyle h_2[Lie]|_{\text{deg }5}}$  & ${\scriptstyle e_2[Lie^{(2)}]|_{\text{deg }5}}$
      &${\scriptstyle W\!H_3}$ \\ 
 ${\scriptstyle \ell=2}$     &${\scriptstyle (4,1)+2(3,2)+3(3,1^2)}$ &${\scriptstyle (5)+2(4,1)+2(3,2)}$ &${\scriptstyle 2(4,1)+2(3,2)+3(3,1^2)}$\\ 
      &${\scriptstyle +2(2^2,1)+2(2,1^3)}$ & ${\scriptstyle +2(3,1^2)+3(2^2,1)+(2,1^3)}$ 
      &${\scriptstyle +2(2^2,1)+(2,1^3)}$\\
      \hline
    ${\scriptstyle (5)}$ & ${\scriptstyle h_1[Lie]|_{\text{deg }5}}$ 
    &${\scriptstyle e_1[Lie^{(2)}]|_{\text{deg }5}}$
    &${\scriptstyle  \omega W\!H_4=Lie_5}$ \\ 
 ${\scriptstyle \ell=1}$   & ${\scriptstyle (4,1)+(3,2)+(3,1^2)}$ 
    &${\scriptstyle (4,1)+(3,2)+(3,1^2)}$ 
    & ${\scriptstyle (4,1)+(3,2)+(3,1^2)}$\\ 
    &${\scriptstyle +(2^2,1)+(2,1^3)}$ &${\scriptstyle +(2^2,1)+(2,1^3)}$ &${\scriptstyle +(2^2,1)+(2,1^3)}$ \\
    \hline
\end{tabular}
\end{center}
\end{ex}
Note from the above example that the two identities in equation (\ref{SymExt}) of Theorem \ref{Compare}, corresponding respectively to (\ref{PBW}) and (\ref{Ext}) below,  themselves yield the following four distinct decompositions of the regular representation, obtained by tensoring each graded piece with the sign representation.  
The decomposition in equation (\ref{Eulerian idempotents}) below is precisely that obtained from the Eulerian idempotents of Gerstenhaber and 
Schack   \cite{GS}; this fact was proved by Hanlon \cite[Theorem 5.1 and Definition 3.6]{Ha}.  Curiously it also appears in a paper of Gessel, Restivo and Reutenauer \cite[Lemma 5.3, Theorem 5.1]{GRR}, where the authors give a combinatorial decomposition of the full tensor algebra as the enveloping algebra of the \textit{oddly generated free Lie superalgebra}; they call equation (\ref{Eulerian idempotents}) below a \lq\lq super\rq\rq version of the Poincar\'e-Birkhoff-Witt theorem.

We have, for $n\geq 1:$
\begin{align}
p_1^n&= \sum_{k\geq 1}\sum_{\stackrel{\lambda\vdash n}{\ell(\lambda)=k}} H_\lambda[Lie]\qquad\qquad \text{(PBW)}\label{PBW}\\
&= \sum_{k\geq 1}\sum_{\stackrel{\lambda\vdash n}{\ell(\lambda)=k}} \omega(H_\lambda[Lie])\qquad \text{(Eulerian idempotents)}\label{Eulerian idempotents}\\
&= \sum_{k\geq 1}\sum_{\stackrel{\lambda\vdash n}{\ell(\lambda)=k}} E_\lambda[Lie^{(2)}]\qquad\quad \text{(Ext)}\label{Ext}\\
&= \sum_{k\geq 1}\sum_{\stackrel{\lambda\vdash n}{\ell(\lambda)=k}} \omega(E_\lambda[Lie^{(2)}])
\end{align}
Example 5.3 shows that these four decompositions are themselves distinct, and also distinct from the two decompositions arising from the Whitney homology of the partition lattice.

We point out one more analogy between $WH_k(\Pi_n)\simeq H^k(Con\!f\, {\mathbb R}^2)$ and the modules $V_k(n)$  arising from the identities of Theorem \ref{Compare}.  In \cite{Su0}, it was shown that the Whitney homology of the partition lattice (and more generally of any  Cohen-Macaulay poset) has the following important property:
\begin{thm}\label{betas}\cite[Proposition 1.9]{Su0} For $0\leq k\leq n-1,$ the truncated alternating sum 
$$W\!H_k(\Pi_n)-W\!H_{k-1}(\Pi_n)+\ldots+(-1)^k W\!H_0(\Pi_n)$$ is a true $S_n$-module, and is isomorphic as an $S_n$-module to the unique nonvanishing homology of the rank-selected subposet of $\Pi_n$ obtained by selecting the first $k$ ranks. Equivalently, 
the degree $n$ term in the plethysm 
$$(e_{n-k}-e_{n-k+1}+\ldots +(-1)^{k} e_n)[Lie]$$ is Schur-positive. In particular, 
the $k$th Whitney homology decomposes into a sum of two $S_n$-modules as follows:
$${\rm ch\,} W\!H_k(\Pi_n)=\omega\left(e_{n-k}[Lie]|_{\text{deg }n}\right)=\beta_n([1,k])+\beta_n([1,k-1]),$$ 
where  $\beta_n([1,k])$ denotes  the Frobenius characteristic of the rank-selected homology of the first $k$ ranks of $\Pi_n$ as in \cite[Proposition 1.9]{Su0}.
\end{thm}

We conjecture that a similar decomposition exists for the $S_n$-modules $V\!h_k(n).$ More precisely, we have 

\begin{conj}\label{LieSup2betas}  Let $V\!h_k(n)$ be the $S_n$-module whose Frobenius characteristic is the degree $n$ term in the plethysm $h_{n-k}[Lie^{(2)}],$ for $k=0,1,\ldots, n-1.$  Then for $0\leq k\leq n-1,$ the truncated alternating sum 
$$V\!h_k(n)-V\!h_{k-1}(n)+\ldots+(-1)^k V\!h_0(n)=U_k(n)$$ is a true $S_n$-module, and hence one has the $S_n$-module decomposition
$$ V\!h_k(n)={\rm ch}^{-1}\,h_{n-k}[Lie^{(2)}]|_{\text{deg }n}\simeq  U_k(n) + U_{k-1}(n).$$
(Here we define $U_{-1}(n)$ to be the zero module and $U_0(n)$ to be the trivial $S_n$-module.
Equivalently, the degree $n$ term in the  plethysm
$$(h_{n-k}-h_{n-k+1}+\ldots +(-1)^{k} h_n)[Lie^{(2)}]$$ is Schur-positive  for $0\leq k \leq n-1.$
\end{conj}
 
This conjecture is easily verified for $0\leq k\leq 3;$ in the latter case there are relatively simple formulas for ${\rm ch\,} V\!h_k(n),$ giving the following for $U_k(n),$ (for $n\geq 4$). Note that if $n=4,$ the last alternating sum vanishes because of the acyclicity property of equation (5.3).
\begin{align*}{\rm ch\,}U_0(n)&={\rm ch\,}V\!h_0(n)=h_n;\\
 {\rm ch\,}U_1(n)&=(h_{n-1}-h_n)[Lie^{(2)}]|_n= h_2h_{n-2}-h_n= s_{(n-1,1)}+s_{(n-2,2)};\\
{\rm ch\,}U_2(n)&={\rm ch\,}V\!h_2(n)-{\rm ch\,}U_1(n)\\ &=Lie_3^{(2)} h_{n-3} + h_2[Lie^{(2)}_2]h_{n-4} -(s_{(n-1,1)}+s_{(n-2,2)})\\
&= h_{n-3}(h_2h_1-h_3) +h_{n-4} (h_4+h_2^2-h_1h_3)-h_{n-2}h_2+h_n\\
&=h_{n-2}s_{(2,1)}-s_{(n-1,1)}-s_{(n-2,2)} +h_{n-4}(h_4+s_{(2,2)}),\\
&\text{which is clearly Schur-positive by the Pieri rule;}
\end{align*}
\begin{align*}
{\rm ch\,}U_3(n)&={\rm ch\,}V\!h_3(n)-{\rm ch\,}U_2(n)\\
&=
h_{n-4}Lie_4^{(2)}+h_{n-5}h_2Lie_3+h_{n-6} h_3[h_2]-{\rm ch\,}U_2(n)\\
&=h_{n-4}(h_4+s_{(2,2)}+s_{(2,1^2)}) +h_{n-5} h_2 s_{(2,1)} +
h_{n-6}(h_6+s_{(4,2)}+s_{(2^3)})\\
&-h_{n-3}s_{(2,1)} -h_{n-4} h_4 -h_{n-4} s_{(2^2)} +s_{(n-1,1})+s_{(n-2,2)}\\
&=h_{n-4} s_{(2,1^2)} +s_{(2,1)} (h_{n-5} h_2-h_{n-3}) +h_{n-6}(h_6+s_{(4,2)}+s_{(2^3)})\\
&+s_{(n-1,1)}+s_{(n-2,2)},
\text{ which is again clearly Schur-positive for }n\geq 5.
\end{align*}
%
\vskip .2in
We include the data for $n=6$ and $n=7$ below.
\vskip .2in
\begin{center}{\small Table 3: Alternating sums $U_k(n)$ of $h_k[Lie^{(2)}]$ for $n=6$}\end{center}
 %
\begin{tabular}{|c|c|}\hline
$k$ & $U_k(6)$\\[3.5pt] \hline
0& {$(6)$}\\[3pt] \hline
1  &${(5,1)}+{(4,2)}$\\[3.5pt] \hline
2  &${(6)}+{(5,1)}+2{(4,2)}
 +{(4,1^2)} +2{(3,2,1)}+{(2^3)}$\\[3.5pt] \hline
3 & ${(6)}+{(5,1)}+3{(4,2)}+2{(4,1^2)}
+{(3^2)}+3{(3,2,1)}+2{(3,1^3)}+2 {(2^2,1^2)}$\\[3.5pt] \hline
4 & $Lie_6^{(2)}={(5,1)}+2{(4,2)}+{(4,1^2)}+3{(3,2,1)}+2{(3,1^3)}+{(2^3)}+ {(2^2,1^2)}+{(2,1^4)}$\\[3.5pt] \hline
\end{tabular}

\begin{center}{\small  Table 4: Alternating sums $U_k(n)$ of $h_k[Lie^{(2)}]$ for $n=7$}\end{center}
\nopagebreak
\begin{tabular}{|c|c|} \hline
$k$ & $U_k(7)$\\ \hline 
0 & ${(7)}$\\ \hline
1 & ${(6,1)}+{(5,2)}$\\ \hline
 2 & $(7)+(6,1)+2(5,2)+(5,1^2) +(4,3)+2(4,2,1)+(3,2^2)$\\ \hline
3 & $(7) +2 (6, 1) +3 (5, 2) +2 (5, 1^2) + 3 (4, 3) +5 (4, 2, 1) 
+2 (4, 1^3)+  2 (3^2, 1)+3 (3, 2^2) $\\
&$ + 3 (3, 2, 1^2)
+ 2 (2^3, 1)$\\  \hline
4 & $2 (6, 1)+ 4 (5, 2)+3 (5, 1^2)+3 (4,3) +8 (4,2,1) +3 (4,1^3)   +  4 (3^2, 1)+  5 (3, 2^2)$\\
&$ + 7 (3, 2, 1^2)
 + 3 (3, 1^4)+ 3 (2^3, 1) +    2 (2^2, 1^3)$\\ \hline
 5 & $ Lie^{(2)}_7=Lie_7=(6, 1)+2 (5, 2)+2 (5, 1^2) + 2 (4, 3) +5 (4, 2, 1) +3 (4, 1^3)$\\
 &$+3 (3^2, 1)+3 (3, 2^2) + 5 (3, 2, 1^2) +2 (3, 1^4)+2 (2^2, 1)+ 2 (2^2, 1^3)   +(2, 1^5) $\\ \hline
\end{tabular}

\vskip .2in

  Recent work of Hyde and Lagarias \cite{HL} rediscovers the representations $\beta_n([1,k])$ of Theorem \ref{betas} in a cohomological setting.
Our results  suggest the existence of a similar topological context in which the modules $V\!h_k(n)$ and $U_k(n)$ appear.

\begin{qn} Is there a cohomological context for the ``$Lie^{(2)}$" identities of Theorem \ref{SymExt}, as there is for the $Lie$ identities in the context of configuration spaces (Theorem \ref{SWThm4.4}), or as in  \cite{HL}?
\end{qn}

Recall from Section 2 and Theorem \ref{EquivPBW} the following facts.  The free Lie algebra has a filtration arising from its derived series \cite[Section 8.6.12]{R}, which in our notation may be described as follows.
Let $\kappa=\sum_{n\geq 2} s_{(n-1,1)}.$  Then     
$Lie_{\geq 2}=\kappa+\kappa[\kappa]+ \kappa[\kappa[\kappa]]+\ldots.$ 

Theorem \ref{SymExt} allows us to  deduce a similar   decomposition for $Lie_n^{(2)}.$  In fact we have the following exact analogue of Theorem \ref{EquivPBW}:

\begin{thm}\label{EquivLieSup2}  The following identities hold, and are equivalent:
\begin{equation}\label{ExtLieSup2}  (E-1)[Lie^{(2)}]=(\sum_{r\geq 1} e_r)[Lie^{(2)}]=\sum_{n\geq 1} p_1^n. \end{equation}
\begin{equation}\label{PlInvLieSup2} (1-H^{\pm})[Lie^{(2)}]=(\sum_{r\geq 1} (-1)^{r-1} h_r)[Lie^{(2)}]=p_1. \end{equation}
\begin{equation}\label{Extge2} (1-H^{\pm})[Lie_{\geq 2}^{(2)}]=(\sum_{r\geq 1} (-1)^{r-1} h_r)[Lie_{\geq 2}^{(2)}]=\omega(\kappa) \end{equation}
Equivalently,
\begin{equation}\label{LieSup2cochain}  \text{ The degree $n$ term in }\sum_{r\geq 0} (-1)^{n-r} h_{n-r}[Lie_{\geq 2}^{(2)}] \text{ is } (-1)^{n-1} s_{(2,1^{n-2})}.
\end{equation}
\begin{equation}\label{LieSup2Filt_a}   Lie_{\geq 2}^{(2)}=Lie^{(2)}[\omega(\kappa)]\end{equation}
\begin{equation} \label{LieSup2Filt_b}
 Lie_{\geq 2}^{(2)}=\omega(\kappa)+\omega(\kappa)[\omega(\kappa)]+\omega(\kappa)[\omega(\kappa)[\omega(\kappa)]]+\ldots  \end{equation}
 \begin{equation}\label{HodgeLieSup2} (E-1)[Lie^{(2)}_{\geq 2}]=\sum_{r\geq 1} e_r[Lie^{(2)}_{\geq 2}]=(1-p_1)^{-1}\cdot H^{\pm}-1
 =\sum_{n\geq 2}\sum_{k=0}^n (-1)^k p_1^{n-k}h_k. \end{equation}

\end{thm}

  We offer two more contrasting results for $Lie_n$ and $Lie_n^{(2)}:$
\begin{prop}\label{AltHLieAltELieSup2}  Let $D\!Par$ be the set of partitions with distinct parts.
\begin{enumerate}[itemsep=8pt]
\item $\sum_{r\geq1} (-1)^{r-1} h_r[Lie]|_{{\rm deg\ }n}=
\begin{cases} p_1 p_2^k, & n=2k+1 \text{ is odd}\\
                    -p_2^k, & n=2k \text{ is even}
                    \end{cases}$
\item     $\sum_{r\geq1} (-1)^{r-1} e_r[Lie^{(2)}]|_{{\rm deg\ }n}=               
\sum_{\stackrel{\lambda\vdash n:\lambda_i=2^{k_i}, k_i\geq 0}{\lambda\in DPar}}
(-1)^{\ell(\lambda)-1} p_\lambda.$

\end{enumerate}
\end{prop}

Next we examine more closely the action on  derangements, i.e. fixed-point-free permutations.
  Reiner and Webb study the Cohen-Macaulay complex of injective words, and compute the $S_n$-action on its top homology \cite{RW}. Theorem \ref{EquivLieSup2} shows that the representations $Lie_n^{(2)}$ make an appearance here as well:



\begin{thm}\label{Lie2Hodge}
Let $n\geq 2.$ For $k\geq 1$ let 
$\Delta_n^k$ denote the degree $n$ term in $e_k[Lie_{\geq 2}^{(2)}].$ Define $\Delta_n=\sum_{k\geq 1} \Delta_n^k  \text{ for } n\geq 2,$ and $ \Delta_1=0, \Delta_0=1.$ 
Then  \begin{enumerate} 
\item  $\Delta_n=\sum_{k=0}^n (-1)^k p_1^{n-k} h_k=p_1\Delta_{n-1}+(-1)^n h_n;$ and hence 
\item For $n\geq 2,$ $\Delta_n$ coincides with the Frobenius characteristic of the homology representation on the complex of injective words in the alphabet $\{1,2,\ldots,n\}.$ 
\end{enumerate}
\end{thm}
\begin{proof} Clearly  $\Delta_n$ is the degree $n$ term in $E[Lie_{\geq 2}^{(2)}],$ so this is nothing but a restatement of equation (\ref{HodgeLieSup2}) above.
\end{proof}

Hanlon and Hersh showed that this homology representation has a Hodge decomposition \cite[Theorem 2.3]{HH}, by showing that the complex itself splits into a direct sum of $S_n$-invariant subcomplexes. Writing $D_n^k$ for the degree $n$ term in $h_k[Lie_{\geq 2}],$ in our terminology their result may be stated as follows:
$$\Delta_n=\sum_{k\geq 1} \omega(D_n^k).$$
In fact  the  identity $\sum_k D_n^k=\sum_{k=0}^n (-1)^k p_1^{n-k} e_k$ is simply a restatement of equation \eqref{LieHodge} in Theorem \ref{EquivPBW}.

Surprisingly, the decomposition of $\Delta_n$ given in Theorem \ref{Lie2Hodge} is different from the Hodge decomposition, i.e. the summands $\Delta_n^k$ and $\omega(D_n^k)$ do not coincide.  The first nontrivial example appears below. 
\begin{ex}  For $n=4,$ we have $\Delta_4=p_1^2h_2-p_1h_3+h_4
=(4)+(3,1)+(2^2)+(2,1^2).$ Also $\Delta_4^2=e_2[h_2]=(3,1),$ 
$\Delta_4^1=Lie_4^{(2)}=(4)+(2^2)+(2,1^2).$ 
The two Hodge pieces, however, each consist of two irreducibles:
$\omega(h_2[Lie_2])=(2^2)+(4)$ and $\omega(h_1[Lie_4])=(3,1)+(2,1^2).$
\end{ex}

This prompts the following:
\begin{qn}  Is there an algebraic complex  explaining the representation-theoretic decomposition 
$$ \Delta_n=\sum_{k\geq 0}\Delta_n^k=\sum_{k\geq 0} e_k[Lie_{\geq 2}^{(2)}]|_{{\rm deg\ } n},$$
just as the Hodge complex explains the decomposition
$$\Delta_n =\omega(\sum_{k\geq 0} h_k[Lie_{\geq 2}]|_{{\rm deg\ } n}),$$
noting (from the preceding example) that $e_k[Lie_{\geq 2}^{(2)}]|_{{\rm deg\ } n}$ is not in general equal to 
$\omega( h_k[Lie_{\geq 2}]|_{{\rm deg\ } n})$?
\end{qn}

Recall from the remarks preceding Corollary \ref{HRRepStab} that equation (\ref{EquivPBWAltExtge2b}) (from the fundamental theorem of equivalences, Theorem \ref{EquivPBW}), when tensored with the sign, can be rewritten as a formula for the alternating sum of Whitney homology modules of $\Pi_n$, when restricted to partitions with no blocks of size 1.  Define 
$W\!H_{\geq 2}^i(\Pi_n)$ to be the sum of  all the homology modules $\tilde{H}(\hat 0, x)$ where $x$ ranges over all partitions into $n-i$ blocks, with no blocks of size 1. Then 
${\rm ch\ } W\!H_{\geq 2}^i(\Pi_n)=\omega( e_{n-i}[Lie_{\geq 2}]|_{{\rm deg\ } n}) $ and so equation (\ref{EquivPBWAltExtge2b}) (and hence the Poincar\'e-Birkhoff-Witt theorem), is equivalent to 
\begin{equation} \sum_{i\geq 0} (-1)^i {\rm ch\,}W\!H_{\geq 2}^i(\Pi_n) = (-1)^{n-1} s_{(2,1^{n-2})}.
\end{equation}
In the notation of \cite{HR}, $\widehat{W}^i_n=W\!H_{\geq 2}^i(\Pi_n)$ (see Corollary 2.11). Hersh and Reiner construct an $S_n$-cochain complex $F_n(A^*)$ with nonvanishing cohomology only in degree $n-1,$ whose $S_n$-character is the irreducible indexed by $(2, 1^{n-2}),$ explicitly proving a conjecture of Wiltshire-Gordon (\cite[Conjecture 1.5, Theorem 1.6, Theorem 1.7]{HR}).

Define, in analogy with \cite{HR}, $\widehat{V_n(k)}$ to be the module with Frobenius characteristic $h_{n-k}[Lie_{\geq 2}^{(2)}]|_{{\rm deg\ }n}.$  Then it is natural to ask:

\begin{qn}  Is there an $S_n$-(co)chain complex for the representations $Lie_{\geq 2}^{(2)}$ whose Lefschetz module  
is given by equation (\ref{LieSup2cochain}) of Theorem \ref{EquivLieSup2} above, i.e. the analogue of equation (2.6)? 
Note that although the nonvanishing (co)homology would occur again only in degree $(n-1)$, affording the same irreducible indexed by $(2, 1^{n-2}),$ the modules in the alternating sum are now different (although they are once again obtained by inducing one-dimensional modules from the same centralisers of $S_n,$ and thus have the same dimensions).  More precisely, and curiously, 
$$\widehat{W}^i_n\not\simeq \widehat{V_n(i)},$$
%
 although in both cases the alternating sums collapse to the irreducible indexed by $(2, 1^{n-2})$.
 For instance, the calculation for  $n=4$ gives:
 $$\omega(e_{1}[Lie_{\geq 2}]|_{{\rm deg\ }4})=\omega(Lie_4)=Lie_4\neq h_{1}[Lie_{\geq 2}^{(2)}]|_{{\rm deg\ }4}=Lie_4^{(2)};$$
 $$\omega(e_{2}[Lie_{\geq 2}]|_{{\rm deg\ }4})=\omega(e_2[e_2])=e_2[h_2]\neq h_{2}[Lie_{\geq 2}^{(2)}]|_{{\rm deg\ }4}=h_2[h_2] .$$

\end{qn}

We summarise these facts in the following:
\begin{thm}\label{HodgeFilt}  We have
\begin{enumerate}
\item (Hodge decomposition for complex of injective words) 
\begin{equation}\label{HodgeInj}
\sum_{r\geq 1} \omega\left(h_r[Lie_{\geq 2}]|_{{\rm deg\ }n}\right)
=\sum_{k=0}^n (-1)^k p_1^{n-k} h_k 
=  \sum_{r\geq 1} e_r[Lie^{(2)}_{\geq 2}]|_{{\rm deg\ }n}  ;
\end{equation}
\item (Derived series filtration)
\begin{equation}\label{LieSupFiltration}
 \sum_{r\geq 1} (-1)^{r-1}\omega\left(e_r[Lie_{\geq 2}]|_{{\rm deg\ }n}  \right)
=s_{(2,1^{n-2})}
=\sum_{r\geq 1} (-1)^{r-1}h_r[Lie^{(2)}_{\geq 2}]|_{{\rm deg\ }n}.
\end{equation}

\end{enumerate}
\end{thm}

Next we see what Part (1) of Theorem \ref{Restrictge2} yields for $F=Lie^{(2)}.$

\begin{prop}\label{Restrictge2LieSup2} Let $\alpha_n=H[Lie^{(2)}_{\geq 2}]|_{{\rm deg\ }n}, n\geq 0$ We have $\alpha_0=1, \alpha_1=0.$
Then $\alpha_n=p_1\cdot \alpha_{n-1} +(-1)^n \sigma_n,$ 
where $\sigma_n=\sum_{i\geq 0}e_{n-2i}g_{2i}.$ Here  
$g_n$ is the virtual representation of dimension zero given by 
$g_n=\sum_\lambda p_\lambda,$ the sum running over all partitions $\lambda$ of $n$ with no part equal to 1, and all parts a power of 2.  In particular $\sigma_n$ is the characteristic of a one-dimensional virtual representation whose restriction to $S_{n-1}$ is $\sigma_{n-1}.$
\end{prop}

The first few virtual representations (symmetric functions) $\sigma_n$ are given below:
\begin{ex}  We have 
\begin{align*} &\sigma_0=1, \sigma_2=e_2+p_2=s_{(2)},\\
&\sigma_3=e_3+e_1p_2=2s_{(1^3)}+s_{(3)},\\
&\sigma_4=2s_{(4)}-s_{(3,1)}+s_{(2^2)},\\
&\sigma_5=2s_{(5)}-s_{(3,1^2)}+s_{(2^2,1)},\\ 
&\sigma_6=2 s_{(6)}+2s_{(4,2)}-2s_{(4,1^2)} -2s_{(3^2)} +s_{(3,1^3)} +2s_{(2^3)} -s_{(2^2, 1^2)}.
\end{align*}
\end{ex}

Note that the analogous recurrence for the exterior powers $E[Lie^{(2)}_{\geq 2}]|_{{\rm deg\ }n}, n\geq 0$, namely, part (2) of Theorem \ref{Restrictge2}, has already been stated in (1) of Theorem \ref{Lie2Hodge}.

As noted at the end of Section 5, in the $Lie$ case, Hersh and Reiner use the corresponding (much simpler) recurrences for $H[Lie_{\geq 2}]|_{{\rm deg\ }n}, n\geq 0$ and $E[Lie_{\geq 2}]|_{{\rm deg\ }n}, n\geq 0$ to derive explicit formulas for the decomposition into irreducibles for each of these representations.  

There is yet another feature of the Lie representation which seems to be shared to some extent by $Lie_n^{(2)}$.
Recall that $Lie_{n-1}\otimes \textbf{ sgn}$ admits a lifting $W_{n}$ which is a true $S_n$-module, the Whitehouse module, appearing in many different contexts \cite{RWh}, \cite{Wh},  whose Frobenius characteristic  is given by 
${\rm ch\,} W_{n}=p_1 \omega(Lie_{n-1}) -\omega(Lie_{n}).$  (See also \cite[Solution to Exercise 7.88 (d)]{St4EC2} for more extensive references.)

One can ask if the same construction for $Lie_n^{(2)}$ yields a true 
$S_{n}$-module.  Clearly one obtains a possibly virtual module which restricts to $Lie_{n-1}^{(2)}$ as an $S_{n-1}$-module. We have the following conjecture, verified in Maple (with Stembridge's  SF package) up to $n=32:$

\begin{conj} The symmetric function $p_1 Lie_{n-1}^{(2)}- Lie_n^{(2)}$ 
is Schur-positive if and only if $n$ is NOT a power of 2.  Equivalently, 
$ Lie_{n-1}^{(2)}\uparrow^{S_n}- Lie_n^{(2)}$ is a true $S_n$-module which lifts $ Lie_{n-1}^{(2)},$ if and only if $n$ is not a power of 2.
\end{conj}

One direction of this conjecture is easy to verify.  Let $n=2^k.$ Then 
$n-1$ is odd, so $Lie_{n-1}^{(2)}=Lie_n.$ Also $Lie_n^{(2)}={\rm ch\,} \textbf{ 1}\uparrow_{C_n}^{S_n}= Conj_n,$ i.e. $Lie_n^{(2)}$ is just the permutation module afforded by the conjugacy action on the class of $n$-cycles of $S_n.$  Consequently it contains the trivial representation (exactly once). But it is well known that $Lie_n$ never contains the trivial representation, and hence, when $n$ is a power of 2,  the trivial module appears with negative multiplicity $(-1)$ in $p_1 Lie_{n-1}^{(2)}- Lie_n^{(2)}.$

We conclude this section with the following observation, which allows us to compute the character values of  $Lie_n^{(2)}$ directly from those of $Lie_n.$  

\begin{thm}  $Lie_n^{(2)}$ is the degree $n$ term in the plethysm
$$\sum_{k\geq 0} Lie[p_{2^k}],$$ and $Lie_n$ is the degree $n$ term in $Lie^{(2)}-Lie^{(2)}[p_2].$  In particular 
$Lie_n^{(2)}=Lie_n$ if $n$ is odd, and coincides with the sign tensored with $Lie_n$ if $n$ is twice an odd number.
\end{thm}
\begin{proof}  This follows from Proposition \ref{Pleth3}, Section 6, since $H[Lie]=E[Lie^{(2)}]$  by Theorem \ref{Compare}.  See also the example following Proposition \ref{Pleth4}.\end{proof}

This yields  the following curious $S_n$-module isomorphism, which gives a recursive definition of $Lie_n^{(2)}:$
\begin{prop}\label{FindingLieSup2}  When $n$ is even:
$$ Lie_n\oplus {\mathbf 1}_{S_2}[Lie_{\frac{n}{2}}^{(2)}]\uparrow_{S_2[S_{\frac{n}{2}}]}^{S_n}\ 
\simeq\,  Lie_n^{(2)} \oplus { \bf sgn}_{S_2}[Lie_{\frac{n}{2}}^{(2)}]\uparrow_{S_2[S_{\frac{n}{2}}]}^{S_n},$$ 
where $S_2[S_{\frac{n}{2}}]$ is the wreath product of $S_2$ with $S_{\frac{n}{2}}$  (i.e. the normaliser of $S_{\frac{n}{2}}\times S_{\frac{n}{2}}$).
If $n$ is odd, this identity simply reduces to the  fact that 
$Lie_n$ and $Lie_n^{(2)}$ coincide.
\end{prop}

The module $Lie_n^{(2)}$ makes an appearance in the decomposition of the module $Conj_n$ of the conjugacy action on the class of $n$-cycles as well.   Again we have the following contrasting results between $Lie$ and $Lie^{(2)}$. 
\begin{thm}\label{ConjLieSup2}
\begin{equation}\label{ConjLiepowsums}\sum_n Conj_n=\sum_{k\geq 1}p_{k}[Lie],\end{equation} 
and hence the sum on the right is Schur-positive.
Equivalently,
\[Lie=\sum_{k\geq 1} \mu(k) p_{k}[Conj]  \]
\begin{equation}\label{ConjLieSup2oddpowsums}\sum_n Conj_n=\sum_{k\geq 1}p_{2k-1}[Lie^{(2)}],\end{equation} 
and hence the sum on the right is Schur-positive.
Equivalently,
\[Lie^{(2)}=\sum_{k\geq 1} \mu(2k-1) p_{2k-1}[Conj]  \]
\end{thm}
\begin{proof} See  equation (\ref{prop6.6eqn1})  of Proposition \ref{prop6.6} for the first identity, and Theorem \ref{thm6.16} for the second identity, in Section 4 of this paper.  The equivalence of the two statements in each case follows by using the plethystic inverse identities (\ref{Pleth8b}), (\ref{Pleth8c}) from Proposition \ref{Pleth8} in Section 6.
\end{proof}

\section{A class of symmetric functions indexed by subsets of primes}

In this section we generalise the definition of $Lie_n^{(2)}$ from Definition \ref{defLieSup2} to obtain  other variants of the representation $Lie_n.$  Although these variants do not share the remarkable properties of $Lie_n^{(2)}$, they do lead to new Schur positivity results; see Theorems \ref{Spositivity1},  \ref{Spositivity2}.  Our main tool here is 
the meta theorem Theorem \ref{metathm} from Section 5.

Recall two $S_n$-modules from Section 2.
In equation~(\ref{definef_n}), we take as a definition that  $Lie_n$ is  $f_n$ with the choice $\psi(d)=\mu(d)$  where $\mu$ is the number-theoretic M\"obius function.  Likewise we define $Conj_n$ to be $f_n$ with the choice $\psi(d)=\phi(d)$  where $\phi$ is the totient function, i.e. Euler's phi-function.  It is well known that $Lie_n$ is the Frobenius characteristic of  the induced representation $\exp(\frac{2i\pi}{n})\uparrow_{C_n}^{S_n},$ where $C_n$ is the cyclic group generated by an $n$-cycle in $S_n$ \cite{R}.  It is also well known that $Conj_n$ is the Frobenius characteristic of the 
induced representation $\textbf{ 1}\uparrow_{C_n}^{S_n},$ or equivalently, the conjugacy action of $S_n$ on the class of $n$-cycles.  

 More generally, a theorem of Foulkes  on the character values of representations induced from the cyclic subgroup $C_n$ of $S_n$ 
asserts Part (1) of the following (see also \cite[Ex. 7.88]{St4EC2}). We refer the reader to \cite{St4EC2} for the definition of the major index statistic on tableaux. 
\begin{thm}\label{Foulkes}   Let $\ell_n^{(r)}$ denote the Frobenius characteristic of the induced representation $\exp\left(\frac{2i\pi}{n}\cdot r\right)\big\uparrow_{C_n}^{S_n},$  $1\leq r  \leq n.$  Then 
\begin{enumerate}
\item \cite{F}
$$\ell_n^{(r)}=\dfrac{1}{n}  \sum_{d|n} \phi(d) \dfrac{\mu(\frac{d}{(d, r)})} {\phi(\frac{d}{(d, r)})} p_d^\frac{n}{d}.$$
\item (\cite{KW}, \cite{St4EC2}) The multiplicity of the Schur function $s_\lambda$ in the Schur function expansion of $\ell_n^{(r)}$ is the number of standard Young tableaux of shape $\lambda$ with major index congruent to $r$ modulo $n.$
\end{enumerate}
\end{thm}


\begin{defn}\label{SetofPrimes} Let $S=\{q_1, \ldots, q_k,\ldots\}$ be a set of distinct primes.  Every positive integer $n$ factors uniquely into $n=Q_n\ell_n$ where $Q_n= \prod_{q\in S} q^{a_q(n)} $ for nonnegative integers $a_q(n),$ and $(\ell_n, q)=1$ for all $q\in S $ such that $a_q(n)\geq 1.$ We associate to the set $S$ a representation $Lie_n^S$, defined via its Frobenius characteristic, as follows:
\begin{equation*}\label{SetofPrimes_a}Lie_n^S= \mathrm{ch}\exp \left(\frac{2\pi i}{n}\cdot Q_n\right){\big\uparrow}_{C_n}^{S_n}
\end{equation*} 
If $\bar{S}$ denotes the set of primes \textit{not} in $S,$ then clearly we have
\begin{equation*}\label{SetofPrimes_b}
Lie_n^{\bar{S}}= \mathrm{ch }\exp \left(\frac{2\pi i}{n}\cdot \ell_n\right)\big\uparrow_{C_n}^{S_n}.\end{equation*}
We may allow $S$ to be the empty set in this definition by interpreting $Q_n$ as the empty product; we then have $Q_n=1$ for all $n\geq 1,$ and hence 
\begin{equation*} Lie_n^\emptyset=Lie_n=\ell_n^{(1)},\quad {\rm  while }\quad
Lie_n^{\bar{\emptyset}}=Conj_n=\ell_n^{(n)};
\end{equation*} thus  $Lie_n^{\bar{\emptyset}}$ is the Frobenius characteristic of $S_n$ acting on the class of $n$-cycles by conjugation. Similarly, at the other extreme, if $S$ is the set of all primes $\mathcal{P}$, then $Q_n=n, \ell_n=1$ for all $n,$ and thus 
\begin{equation*} 
Lie_n^{\mathcal{P}}=Conj_n=\ell_n^{(n)} \quad {\rm  while }\quad Lie_n^{\bar {\mathcal{P}} }=Lie_n=\ell_n^{(1)}.
\end{equation*}
Let $P(S)$ denote the set of positive integers whose  prime divisors constitute a subset of the set of primes $S;$  note that $1\in P(S).$ Similarly let $P(\bar{S})$ be the set of positive integers whose set of prime divisors is disjoint from  $S;$ note that $P(\bar{S})$ is precisely the set of integers that are relatively prime to every prime in $S.$   Also $P(S)\cap P(\bar{S})=\{1\}.$  Equivalently, 
\begin{equation*}P(S)=\{n\geq 1: q \text{ is a prime factor of }n \Longrightarrow q\in S \},\end{equation*}
\begin{equation*}P(\bar{S})=\{n\geq 1: q \text{ is a prime factor of }n \Longrightarrow q\notin S\}.\end{equation*}

\end{defn}

When $S$ consists of a single prime $\{q\}$, we can view the modules $Lie_n^{(q)}$ as interpolating between the modules $Lie_n$ and $Conj_n=\text{ch } 1\uparrow_{C_n}^{S_n},$ since:
\begin{equation}
 Lie_n^{(q)}=Lie_n \text{ when $n$ is relatively prime to the prime $q,$}  
\end{equation}
and, at the other extreme:
\begin{equation}
Lie_n^{(q)}=Conj_n=\text{ch } 1\uparrow_{C_n}^{S_n} \text{ when $n$ is a power of the prime $q.$ }
\end{equation}

More generally, if $S$ is a nonempty set of primes, then 
\begin{equation} Lie_n^S=
\begin{cases} Lie_n, & n\in P(\bar{S})\\
                  Conj_n, &n\in P(S).
\end{cases}
\end{equation}
 
Our goal is to describe the symmetric and exterior powers of these 
modules, the analogues of the higher Lie modules in Section 2. We will apply the meta theorem, Theorem \ref{metathm}, of Section 5. In order to do this we must first determine the values $Lie_n^S(\pm 1)$ for each subset $S$ of primes.

\begin{lem}\label{Ramanujan} Let $r$ be a divisor of $n$ such that $(r, \frac{n}{r})=1.$ Consider the representation with Frobenius characteristic 
$$\ell_n^{(r)} =\mathrm{ch}\exp \left(\frac{2\pi i}{n}\cdot r\right)\big\uparrow_{C_n}^{S_n}.$$
Then $\ell_n^{(r)}(1)=$
$\begin{cases} 1, & n=r\\
                    0, &\text{ otherwise.}
\end{cases}$
%
%
\end{lem}
\begin{proof} For ease of notation we write $n=rs$ where $r,s$ are relatively prime positive integers.  Then every divisor $d$ of $n$ factors uniquely into $d=r_1 s_1$ where $r_1$ is a divisor of $r$, 
$s_1$ is a divisor of $s,$ and $r_1,$ $s_1$ are thus relatively prime. 
In particular $(d,r)=r_1,$ and $\frac{d}{(d,r)}=s_1.$  Hence we have 
\begin{align*} \ell_n^{(r)}(1)&=\frac{1}{n}\sum_{d|n}\phi(d) \frac {\mu(\frac{d}{(d,r)})}{\phi(\frac{d}{(d, r)})}
=\frac{1}{n}\sum_{{r_1|r},{s_1|s}} \phi(r_1 s_1) 
\frac{\mu(\frac{d}{r_1})}{\phi(\frac{d}{r_1})}\\
&=\frac{1}{n}\sum_{{r_1|r},{s_1|s}} \phi(r_1 s_1) 
\frac{\mu(s_1)}{\phi(s_1)}
=\frac{1}{n}\sum_{{r_1|r},{s_1|s}} \phi(r_1) 
\mu(s_1)\\
&=\frac{1}{n}\left(\sum_{r_1|r} \phi(r_1)\right )\left(\sum_{s_1|s}\mu(s_1)\right)=\frac{r}{n}\sum_{s_1|s}\mu(s_1).
\end{align*}
The last sum is nonzero if and only if $s=\frac{n}{r}=1,$ i.e. $n=r,$ and the result follows.
\end{proof}

\begin{cor}\label{LSvalues} Let $S$ be a fixed  set of primes, and let $Lie_n^S$ and $Lie_n^{\bar{S}}$ be (the Frobenius characteristics of) the representations of Definition \ref{SetofPrimes}.  Then 
\begin{enumerate} \item $Lie_n^S(1)$ is nonzero if and only if $n=1$ or the distinct prime factors of $n$ are contained in the set $S,$ i.e. if and only if $n\in P(S),$ in which case it equals 1. 
\item $Lie_n^{\bar{S}}(1)$ is nonzero if and only if $n$ is relatively prime to every prime in the set $S,$ i.e. if and only if $n\in P(\bar{S}),$ in which case it equals 1.
\item In particular $Lie_n^\emptyset(1)=Lie_n(1)=Lie_n^{\bar{\mathcal{P}}}(1)$ is nonzero unless $n=1,$ in which case it equals 1.  Similarly $Lie_n^{\bar{\emptyset}}(1)=Conj_n(1)$ equals 1 for all $n\geq 1.$
\end{enumerate}
\end{cor}
\begin{proof} Write $n=Q_n\ell_n$ where $Q_n$ is the product of all maximal prime powers $q_i^{a_i}$ that are divisors of $n,$ and $\ell_n$ is relatively prime to every prime $q_i$ in $S.$   In Lemma \ref{Ramanujan}, take $r=Q_n$ for Part (1) and $r=\frac{n}{Q_n}$ for Part (2). 
Lemma \ref{Ramanujan} says $Lie_n^S$ (respectively $Lie_n^{\bar{S}}$) is nonzero  (and then equal to 1) if and only if $Q_n=n$ (respectively $Q_n=1$).  Note that this applies also when $S$ is empty (respectively $S=\mathcal{P}$), since then $Q_n=1$ (respectively $Q_n=n$). The claims follow immediately from the observations succeeding Definition \ref{SetofPrimes}.
\end{proof}

The next step is to apply the meta theorem, Theorem \ref{metathm}, to the sequences of symmetric functions $f_n=Lie_n^S$ and $f_n=Lie_n^{\bar S}.$   In order to do this, we must verify that each sequence is determined by a single function $\psi$ whose value depends only on the set $S.$   (See equation~(\ref{definef_n}).) 
Using Theorem \ref{Foulkes}, we see that with $Q_n, \ell_n$ as in Definition \ref{SetofPrimes}, one has 
\begin{center}$Lie_n^S=\frac{1}{n}\sum_{d|n} \psi_n(d) p_d^{\frac{n}{d}}$\quad
with \quad
$\psi_n(d)=\phi(d) \dfrac{\mu(d/(d,Q_n))}{\phi(d/(d,Q_n))}.$\end{center}
($\psi_n(d)$ is  a \textit{Ramanujan sum}.)
We claim that this expression is in fact independent of $n.$ 
Since $d$ divides $n,$ we can factor $d$ uniquely as $d=Q_d\ell_d$ where $\ell_d$ is relatively prime to $Q_d$ as well as to all the primes in the set $S.$ 
In particular $(d, Q_n)=Q_d.$ 
Using the multiplicative property of $\phi,$ we obtain
$$\psi_n(d)=\phi(Q_d)\phi(\ell_d) \dfrac{\mu(d/Q_d)}{\phi(d/(Q_d))}
=\phi(Q_d)\phi(\ell_d) \dfrac{\mu(\ell_d)}{\phi(\ell_d)}
=\phi(Q_d)\mu(\ell_d),$$
thereby showing that $\psi_n(d)=\psi(d)$ depends only on $d$ and the set of primes $S.$ 
An entirely analogous calculation shows that 
\begin{center}$Lie_n^{\bar S}=\frac{1}{n}\sum_{d|n} {\bar \psi}(d) p_d^{\frac{n}{d}}$\quad
with\quad  $\bar\psi(d)=\phi(\ell_d)\mu(Q_d).$\end{center} 
In summary, given a set of primes $S$ and a positive integer $d,$ if $d$ is factored uniquely as $d=Q_d\ell_d$ where $\ell_d$ is relatively prime to the primes in $S$ and also to $Q_d,$ (so that $Q_d$ is a product of prime powers of elements of $S$), then 
\begin{equation}\label{FoulkesPrimes1} \text{For all }n \geq 1,\quad
Lie_n^S=\frac{1}{n}\sum_{d|n} \psi(d) p_d^{\frac{n}{d}} \quad 
\text{ with } \psi(d)=\phi(Q_d)\mu(\ell_d), \text{ and}
\end{equation}
\begin{equation}\label{FoulkesPrimes2} \text{For all }n \geq 1,\quad
Lie_n^{\bar S}=\frac{1}{n}\sum_{d|n} \bar\psi(d) p_d^{\frac{n}{d}} \quad
\text{ with } \bar\psi(d)=\phi(\ell_d)\mu(Q_d).
\end{equation}


With the calculation of $Lie_n^S(1)$ from Corollary~\ref{LSvalues}, we are ready to invoke the power of Theorem~\ref{metathm}.
From equations \eqref{metaSym} and \eqref{metaAltExt}, we have the following formula for the higher $Lie^S$-modules:
\begin{thm}\label{SymAltSymLS} Let $Lie^S=\sum_{n\geq 1} Lie_n^S. $   
 Then one has the generating functions
   \begin{equation}\label{SymLS}
 H[Lie^S](t) = \prod_{n\in P(S)}  (1-t^n p_n)^{-1}
=\sum_{\lambda\in Par:\lambda_i\in P(S)} t^{|\lambda|} p_\lambda;
\end{equation}
\begin{equation}\label{AltSymLS} H[\omega(Lie^S)^{alt}](t)=\prod_{n\in P(S)} (1+t^n p_n)
=\sum_{\lambda\in D\!Par:\lambda_i\in P(S)} t^{|\lambda|} p_\lambda;
\end{equation}
\end{thm}

Recall the definitions of $P(S), P(\bar{S})$ from Definition \ref{SetofPrimes}.
\begin{thm}\label{Spositivity1} Fix a set $S$ of primes.  The following sums of power sums are Schur-positive:
\begin{enumerate}
\item $\sum_{\lambda\vdash n: \lambda_i\in P(S)} p_\lambda;$
\item $\sum_{\lambda\vdash n: \lambda_i\in P(\bar{S})} p_\lambda;$
\item $\sum_{\stackrel{\lambda\vdash n: \lambda_i\in P({S})}{n-\ell(\lambda) \text{ even}}} p_\lambda=\sum_{\stackrel{\lambda\vdash n: \lambda_i\in P({S})}{\lambda 
 \text{ has an even number of even parts}}} p_\lambda.$
\end{enumerate}
\begin{proof} Parts (1) and (2) are immediate since $Lie_n^S$ and $Lie_n^{\bar{S}}$ are Schur-positive, so the left side of (\ref{SymLS}) in each case is a symmetric power of  a true $S_n$-module.
For Part (3), observe first that  $n-\ell(\lambda)$ is congruent to the number of even parts of $\lambda$, and hence Part (3) coincides with Part (1) unless $2\in S.$ In that case, applying the involution $\omega$ to 
equation (\ref{SymLS}) of Theorem \ref{SymAltSymLS}, (with $t=1$), we have 
\begin{equation*} 
\sum_{\lambda\vdash n: \lambda_i\in P({S}), n-\ell(\lambda) \text{ even}} p_\lambda
=\frac{1}{2} (H[Lie^S]+\omega(H[Lie^S]))
\end{equation*}
But the Schur expansion on the left-hand side has integer coefficients, and the Schur expansion on the right-hand side is certainly positive.  
Part (3) follows.\end{proof}
\end{thm}

In the special case when the set $S$ is the set of all odd primes,  equation (\ref{SymLS})  and Part (2) of Theorem \ref{Spositivity1}   describe the sum of power sums $p_\lambda$ for all parts of $\lambda$ odd, as the symmetrised module of a representation, whereas  \cite[Theorem 4.9]{Su1} (see also Theorem \ref{ExtAltExtLS} below) gives a description as  the exterior power of the conjugacy action.

When $S=\emptyset, $   equations (\ref{SymLS}) and (\ref{AltSymLS}) applied to $S$ and $\bar{S}$ reduce to  known formulas of Thrall \cite{T}, Cadogan \cite{C}, and Solomon \cite{So}, respectively.  
See also \cite[Ex. 7.71, Ex. 7.88, Ex. 7.89]{St4EC2}
\begin{thm} \label{ThrallPBWCadoganSolomon}
\begin{equation*}\label{ThrallPBW}  (Thrall,\ PBW) \qquad H[\sum_{n\geq 1} Lie_n](t)=(1-tp_1)^{-1}
\end{equation*}
 \begin{equation*}\label{Cadogan} (Cadogan)\qquad  H[\sum_{n\geq 1} (-1)^{n-1} \omega(Lie_n)](t)
=1+tp_1. \end{equation*} 
\begin{equation*}\label{Solomon} (Solomon )\qquad H[\sum_{n\geq 1}  \mathrm{ch}
(1\big\uparrow_{C_n}^{S_n})](t)=\prod_{n\geq 1} (1-t^np_n)^{-1}
\end{equation*}
\end{thm}

Next we compute the values of $Lie_n^S(t)$ for $t=-1.$  Thanks to Proposition~\ref{metaf}, we can avoid another cumbersome Ramanujan sum computation.
\begin{lem}\label{LSnegvalues}  For the set of distinct primes 
$S:$
\begin{enumerate}
\item If $2\notin S,$ then $Lie_n^S(-1)=
\begin{cases} -1, & n\in P(S)\\
                       1, &\frac{n}{2}\in P(S)\\
                       0, &\text{ otherwise.}\\
\end{cases}$
\item If $2\in S,$ then $Lie_n^S(-1)=
\begin{cases} -1, &n \text{    odd and } n\in P(S)\\
                       0, &\text{ otherwise.}\\
\end{cases}$
\item  $Lie_n^\emptyset(-1)=Lie_n^{\bar{\mathcal{P}}}(-1)
=\begin{cases} -1, & n=1\\
                       1, &n=2\\
                       0, &\text{ otherwise.}\\
\end{cases}$
\item    $Lie_n^{\bar{\emptyset}}(-1)=Lie_n^{\mathcal{P}}(-1)=
\begin{cases} -1, & n\text{    odd }\\
                     0, &\text{ otherwise.}\\
\end{cases}$
\end{enumerate}
\end{lem}
\begin{proof} It suffices to address Parts (1) and (2). 
If $n$ is odd, by Proposition \ref{metaf} and Corollary \ref{LSvalues}, we have 

$Lie_n^S(-1)=-Lie_n^S(1)=\begin{cases} -1, &n\in P(S)\\ 0, &otherwise. \end{cases}$

If $n$ is even, by Proposition \ref{metaf}, we have
 $Lie_n^S(-1)=Lie_{\frac{n}{2}}^S(1)-Lie_n^S(1).$  Invoking Part (1) of Corollary \ref{LSvalues}, we see that, when $2\in S,$  $Lie_n^S(1)=1$ if and only if $Lie_{\frac{n}{2}}^S(1)=1.$  Hence when $2\in S$ and $n$ is even, $Lie_n^S(-1)=0.$

Now assume $n$ is even and $2\notin S.$ Then $n\notin P(S)$ so $Lie_n^S(1)=0.$ In this case $Lie_{\frac{n}{2}}^S(1)\neq 0$ if and only if $\frac{n}{2}\in P(S),$ in which case $Lie_n^S(-1)=Lie_{\frac{n}{2}}^S(1)=1.$

This establishes Parts (1) and (2), and hence the remaining parts.  \end{proof}
As an immediate consequence, invoking equations (\ref{metaExt}) and (\ref{metaAltSym}) of Theorem \ref{metathm}, we obtain the analogue of Theorem \ref{SymAltSymLS}.  
\begin{thm} \label{ExtAltExtLS}Let $S$ be a set of  primes, and let $Lie_n^S$ 
be as defined in Theorem \ref{SymAltSymLS}.  Then 
 \begin{equation}\label{ExtLS} E[Lie^S](t)=
\begin{cases} \prod_{ n\in P(S)} (1-t^n p_n)^{-1} \prod_{n\, even, \frac{n}{2}\in P(S)} (1-t^np_n), & 2\notin S\\
\prod_{n\, odd, \, n\in P(S) } (1-t^np_n)^{-1}=\prod_{n\in P(S\backslash\{2\}) } (1-t^np_n)^{-1}, & 2\in S.
\end{cases}
\end{equation}
(Note $n$ is necessarily odd in the first product of the case $2\notin S.$)
\begin{equation}\label{omegaExtLS} \omega(E[Lie^S])(t)=
 \prod_{ n\in P(S)} (1-t^n p_n)^{-1}\prod_{n\, even, \frac{n}{2}\in P(S)} (1+t^np_n),  \text{ provided } 2\notin S
\end{equation}
\begin{equation}\label{AltExtLS} E[\omega(Lie^S)^{alt}](t)=
\begin{cases} \prod_{ n\in P(S)} (1+t^n p_n) \prod_{n\, even, \frac{n}{2}\in P(S)} (1+t^np_n)^{-1}, & 2\notin S\\
\prod_{n\, odd, \, n\in P(S) } (1+t^np_n), & 2\in S.
\end{cases}
\end{equation}
\end{thm}
\begin{proof} The second equation is clearly a consequence of applying the involution $\omega$ to the first, which is obtained directly from equation (\ref{metaExt}) of Theorem \ref{metathm}.  Similarly the third equation  follows directly from (\ref{metaAltSym}).  
\end{proof}

As in Theorem \ref{SymAltSymLS}, when $S=\emptyset,$ the first and third equations above allow us to  recover the formulas of \cite[Corollary 5.2]{Su1} and \cite[Theorem 4.2]{Su1}, which we restate using the equivalent formulations of the second and the third equations, in order to emphasise the fact that the right-hand side may be written as a nonnegative linear combination of power sums.

Recall that we write 
$Conj_n=Lie_n^{\mathcal{P}}=\mathrm{ch} (1\big\uparrow_{C_n}^{S_n}).$ 
\begin{prop}\label{ExtLieConj} \cite[Theorem 4.2 and Corollary 5.2]{Su1} We have the generating functions:
\begin{enumerate}[itemsep=8pt]
\item $\omega(E[Lie](t))=(1+t^2 p_2)(1-tp_1)^{-1}$
\item $E[\sum_{n\geq 1} Conj_n](t)=
\prod_{n\geq 1,\, n\, odd} (1-t^np_n)^{-1}$
\item 
$\sum_{\lambda\in Par} (-1)^{|\lambda|-\ell(\lambda)} H_\lambda[Lie] (t)=\omega(E[\omega(Lie)^{alt}])(t)=(1+tp_1)(1-t^2p_2)^{-1}$
\item 
$E[\sum_{n\geq 1} (-1)^{n-1} \omega(Conj_n)](t)=
\prod_{n\geq 1,\, n\, odd} (1+t^np_n)$
\end{enumerate}
\end{prop}

For any subset $T_n$ of partitions of $n,$ denote by $P_{T_n}$ the sum of power-sum symmetric functions $\sum_{\lambda\in T_n} p_\lambda.$
Note the passage from symmetric to exterior powers, for the same representation, in Part (1) below.
\begin{thm}\label{Spositivity2} For a fixed set of primes $S,$ the  sums $P_{T_n}$ are Schur-positive 
for the following choices of $T_n$:
\begin{enumerate}
\item If $2\in S,$ $T_n=\{\lambda\vdash n: \lambda_i\, odd,\, \lambda_i\in P(S)\};$ in this case the representation coincides with 
$$H[Lie^{S\backslash\{2\}}]=E[Lie^S]=\sum_{\lambda:\lambda_i \in P(S\backslash\{2\})} p_\lambda.$$
\item 
If $2\notin S,$ $T_n$ consists of all partitions $\lambda$ of $n$ such that the parts are (necessarily odd and) in $P(S),$ or the parts are twice an odd number in $P(S)$, the even parts occurring at most once.
\end{enumerate}
\end{thm}
\begin{proof}  Part (1) has also been observed in Theorem \ref{Spositivity1}, (2), taking the set of primes to be $S\backslash\{2\}$.  Here it follows from equation (\ref{ExtLS}) of Theorem \ref{ExtAltExtLS}.  In particular we have two ways to realise the representation 
$$\sum_{\lambda:\lambda_i \in P(S\backslash\{2\})} p_\lambda,$$
namely as the symmetric power $H[Lie^{S\backslash\{2\}}]$ (Theorem \ref{SymAltSymLS}) and also as the exterior power $E[Lie^S],$ the latter from Theorem \ref{ExtAltExtLS}.

Part (2) follows by extracting the degree $n$ term from the third equation in Theorem \ref{ExtAltExtLS}, noting that we now have exterior powers of Schur positive functions.
\end{proof}

In particular when the set $S$ consists of a single prime $q,$ we have:

\begin{thm}\label{OnePrime} Let $q$ be a fixed prime.  The following  multiplicity-free sums of power sums are the Frobenius characteristics of a true represention of $S_n,$ of dimension $n!$
\begin{equation*}
\sum_{{\lambda\vdash n}
\atop{ \text{all parts of } \lambda\text{\ are  powers of\ } q} }
p_\lambda
\end{equation*}
(If $q$ is an odd prime,  this representation is self-conjugate.)
\begin{equation*} 
\sum_{\substack{\lambda\vdash n\\(\lambda_i, q)=1\text{ for all }i} }p_\lambda;
\end{equation*}
\begin{equation*} 
\sum_{\substack{\lambda\vdash n,\lambda_i \text{ odd }\\(\lambda_i, q)=1\text{ for all }i}} p_\lambda, \qquad q \text{ odd}.
\end{equation*}

\end{thm}

With $\ell_n^{(k)}$ denoting the characteristic ${\rm ch}\, (\exp(\frac{2\pi i k}{n})\uparrow_{C_n}^{S_n})$ of the Foulkes character, we have the following contrasting result:
\begin{thm}\label{Su1Thm5.6}\cite[Lemma 5.5, Theorem 5.6]{Su1} If $k$ is any fixed positive integer, then $\ell_n^{(k)}(1)$ is nonzero if and only if $n$ divides $k,$ in which case it equals 1.  Hence one has the Schur-positive sum 
\begin{equation*}\sum_{\lambda\vdash n: \lambda_i|k} p_\lambda
=H[\sum_{m\geq 1} \ell_m^{(k)} ]\vert_{{\rm\, deg\,} n}.
\end{equation*}
\end{thm}

We conclude this section by recording the results for the important special case of the single prime 2:

\begin{thm}\label{PlInv-E} Let $Lie^{(2)}_n$ be the Frobenius characteristic of the induced representation $\exp(\frac{2i\pi}{n}\cdot 2^k)\large\uparrow_{C_n}^{S_n},$  where $k$ is the largest power of 2 which divides $n.$ Then we have the following generating functions:
\begin{enumerate}[itemsep=8pt]
\item \label{ExtLieSup2Reg}\qquad\qquad\qquad\qquad\qquad
$ E\left[\sum_{n\geq 1}  Lie^{(2)}_n\right ](t) = (1-tp_1)^{-1}.$

\text{Equivalently,}
\qquad\quad
$H^{\pm}\left[\sum_{n\geq 1} Lie^{(2)}_n\right](t)=1-tp_1.$

\item\label{PlInv-ELieSup2}$E\left[\sum_{n\geq 1} (-1)^{n-1}\omega(Lie^{(2)}_n)\right](t)=1+tp_1.$

That is, the  plethystic inverse of $\sum_{n\geq 1} e_n$ is given by 
\begin{center}$\sum_{n\geq 1} (-1)^{n-1}\omega(Lie^{(2)}_n).$
\end{center}

\item\label{SymLieSup2} $H\left[\sum_{n\geq 1}  Lie^{(2)}_n\right](t)
=\prod_{n=2^k, k\geq 0} (1-t^n p_n)^{-1}
=\sum_{{\lambda\in Par}\atop{\text{every part is a power of } 2}} t^{|\lambda|}p_\lambda;$
\item\label{SymAltLieSup2} $\sum_{\lambda\in Par} (-1)^{|\lambda|-\ell(\lambda)} \omega(E_\lambda[Lie^{(2)}]) (t) $

$=H\left[\sum_{n\geq 1} (-1)^{n-1}\omega(Lie^{(2)}_n)\right](t)=\prod_{n=2^k, k\geq 0} (1+t^n p_n)$

$ =\sum_{{\lambda\in D\!Par}\atop{ \text{every part is a power of } 2}}t^{|\lambda|} p_\lambda$
\end{enumerate}
\end{thm}
\begin{proof}  Recall from Section 1 that $DPar$ is the set of partitions with all parts distinct.  The identities follow respectively from  Theorem \ref{ExtAltExtLS} and Theorem \ref{SymAltSymLS}.  It suffices to observe that, because $P(S)=\{2^k:k\geq 0\}:$
\begin{enumerate}
\item Corollary \ref{LSvalues} specialises in this case to give $Lie_n^{(2)}(1)=1 $ if and only if $n$ is a power of 2, and zero otherwise;
\item Part (2) of Lemma \ref{LSnegvalues} specialises to give $Lie_n^{(2)} (-1) = -1$ if $n=1$ and zero otherwise.
\end{enumerate}
For the equivalence of the second equation in (1) and the equation (2), we invoke Lemma \ref{pmalt},  which also gives the equivalence of (3) and (4).  The latter equation also follows by applying \ref{metaAltExt} in Theorem \ref{metathm} directly.
\end{proof}

We can now supply proofs for the new theorems in Section 2.
\vskip .2in
\noindent
{\bf Proof of Theorem \ref{Compare}:}

\begin{proof} Specialise Theorem \ref{SymAltSymLS} and Theorem \ref{ExtAltExtLS} to the cases $S=\emptyset$ (for the $Lie$ identities) and $S=\{2\}$ (for the $Lie^{(2)}$ identities).

The $Lie^{(2)}$  identities are all restatements of Theorem \ref{PlInv-E}, using the definition of $H_\lambda$ and $E_\lambda.$  Likewise the $Lie$ identities are all restatements of Theorem \ref{ThrallPBWCadoganSolomon} and Theorem \ref{EquivPBW}.

The first equation in (\ref{TotalCoh}) is, for instance,  a restatement of the first equation of Proposition \ref{ExtLieConj}. 

The statement about the equivalence of the $Lie$ (respectively, $Lie^{(2)}$) identities is a consequence of  Proposition \ref{metaequiv}.
%
\end{proof}

\vskip .2in
\noindent
{\bf Proof of Theorem \ref{EquivLieSup2}:}
\begin{proof}  Equation (\ref{ExtLieSup2}) is the identity  $E[Lie^{(2)}]=(1-p_1)^{-1}$ of Theorem \ref{PlInv-ELieSup2}, and hence by Lemma \ref{pm}, we obtain $H^{\pm}[Lie^{(2)}]=1-p_1,$ which is equation (\ref{PlInvLieSup2}) after removing the constant term and adjusting signs.

Now invoke (\ref{metage2d}) of Theorem \ref{metage2}, Section 5.  We have 
$$H^{\pm}[Lie_{\geq 2}^{(2)} ]=E\cdot (1-p_1)
=1+ \sum_{n\geq 2} (e_n-e_{n-1} p_1)
= 1-\omega(\kappa), $$
and this is precisely equation (\ref{Extge2}), again after cancelling the constant term and adjusting signs.  The remaining identities and equivalences follow exactly as in Theorem \ref{EquivPBW}, Section 5.
\end{proof}
\vskip.2in
\noindent 
{\bf Proof of Proposition \ref{AltHLieAltELieSup2}:}
\begin{proof}
Part (1) follows from Part 3 of Proposition \ref{ExtLieConj}, and Part (2) follows from Part 4 of Theorem \ref{PlInv-E}.
\end{proof}

\noindent
{\bf Proof of Proposition \ref{Restrictge2LieSup2}:}
\begin{proof}  We apply Theorem \ref{Restrictge2} (1), Section 5, using Part (3)  of Theorem \ref{PlInv-E}.  Then clearly $g_n$ is as stated.
\end{proof}

\section{A class of symmetric functions indexed by subsets of primes}

In this section we consider variants of the representation $Lie_n.$  We take as a definition that  $Lie_n$ is  $f_n$ with the choice $\psi(d)=\mu(d)$ in equation (\ref{definef_n}), where $\mu$ is the number-theoretic M\"obius function.  Likewise we define $Conj_n$ to be $f_n$ with the choice $\psi(d)=\phi(d)$  where $\phi$ is the totient function, i.e. Euler's phi-function. It is well known that $Lie_n$ is the Frobenius characteristic of the action of $S_n$ on the multilinear component of the free Lie algebra, and also of the induced representation $\exp(\frac{2i\pi}{n})\uparrow_{C_n}^{S_n},$ where $C_n$ is the cyclic group generated by an $n$-cycle in $S_n$ (see \cite{R} for background and history.) It is also well known that $Conj_n$ is the Frobenius characteristic of the 
induced representation $\textbf{ 1}\uparrow_{C_n}^{S_n},$ or equivalently, the conjugacy action of $S_n$ on the class of $n$-cycles.

Finally we address the third property of $Lie_n^{(2)}$ mentioned at the beginning of Section 2. It  can be deduced from the following more general fact:

Let $\omega_n=\exp\frac{2\pi i}{n},$ and let $\chi^r$ be the representation of $C_n$ afforded by $\omega_n^r.$
\begin{prop} The following isomorphisms hold: 
\begin{enumerate}
\item
$${\bf sgn}_{S_n}\otimes (\chi^r\uparrow_{C_n}^{S_n})\simeq \chi^{r+{n\choose 2}}\uparrow_{C_n}^{S_n}.$$
\item
$$\chi^{r+{n\choose 2}}\uparrow_{C_n}^{S_n}\simeq \chi^{{n\choose 2}-r}\uparrow_{C_n}^{S_n}.$$
\end{enumerate}
Hence $\omega(Lie_{2(2k+1)}^{(2)})=Lie_{2(2k+1)}.$
\begin{proof}  Let $\chi^\mu$ denote the $S_n$-irreducible indexed by the partition $\mu$ of $n.$ Then $\chi^\mu\otimes \textbf{sgn}=\chi^{\mu'}.$  Also note that the sign representation restricted to the cyclic group $C_n$ sends the generator to $(-1)^{n-1}.$  Thus we have 
\begin{align*}
\langle \chi^\mu, \textbf{sgn}_{S_n}\otimes (\chi^r\uparrow_{C_n}^{S_n})\rangle &=\langle \chi^\mu\otimes \textbf{sgn},  (\chi^r\uparrow_{C_n}^{S_n})\rangle \\
&=\langle (\chi^\mu\otimes \textbf{ sgn})\downarrow_{C_n},  \chi^r\rangle_{C_n} \\
&=\langle \chi^\mu\downarrow_{C_n}\otimes \textbf{ sgn}\downarrow_{C_n},  \chi^r\rangle_{C_n} \\
&  =\langle \chi^\mu\downarrow_{C_n},  \textbf{ sgn}\downarrow_{C_n}\otimes  \chi^r\rangle_{C_n} \\
&=\langle \chi^\mu\downarrow_{C_n},  (-1)^{n-1}\cdot \chi^r\rangle_{C_n} \\
&=\langle \chi^\mu,  ((-1)^{n-1}\cdot \chi^r)\uparrow^{S_n}\rangle_{S_n}.
\end{align*}
Hence we have 
$$\textbf{sgn}_{S_n}\otimes (\chi^r\uparrow_{C_n}^{S_n})\simeq 
((-1)^{n-1}\cdot \chi^r)\uparrow^{S_n}.$$
But $(-1)^{n-1}=\exp\pi\cdot (n-1),$ and thus 
$$(-1)^{n-1}\cdot \chi^r=\exp\left(\frac{2\pi i}{n} {\cdot (r+(n-1)\frac{n}{2}})\right),$$ 
as claimed.

For Part (2), observe that 
$$\chi^r\uparrow_{C_n}^{S_n}\simeq \chi^{-r}\uparrow_{C_n}^{S_n}$$
this follows, for example, from the formula for an induced representation, using the fact that  a permutation and its inverse are conjugate in $S_n$.  More precisely one shows the character value equality 
$$\chi^{-r}\uparrow^{S_n}(\sigma)=\chi^{r}\uparrow^{S_n}(\sigma^{-1})
=\chi^{r}\uparrow^{S_n}(\sigma).$$
Hence we have 
$$\chi^{r+{n\choose 2}}\uparrow_{C_n}^{S_n}\simeq \chi^{-{n\choose 2}-r}\uparrow_{C_n}^{S_n}\simeq \chi^{-{n\choose 2}-r}
\uparrow_{C_n}^{S_n}.$$
But $\chi^{-{n\choose 2}-r}=\chi^{-{n\choose 2}-r}\otimes \chi^{n(n-1)}
=\chi^{{n\choose 2}-r}$ since $\chi^{n(n-1)}=1.$
\end{proof}
\end{prop}

\section{A formula for $\prod_{n\in T} (1-p_n)^{-1}$} 

In this section we explore, for a fixed subset $T$ of positive integers,   the sum of power sums resulting from the product 
$\prod_{n\in T} (1-p_n)^{-1}.$
(This product is 1 if $T$ is the empty set.)

The results of this section were announced in \cite{SuFPSAC2019}.

\begin{defn}\label{def6.1} Fix a nonempty subset $T$ of the positive integers.  For each positive integer $d,$ define a function $\psi^T$ by $\psi^T(d)=\sum_{m|d,\, m\in T} m\,\mu\left(\frac{d}{m}\right).$
\end{defn}
\begin{defn}\label{def6.2} For each nonempty subset $T$ of positive integers, 
 define a sequence of (possibly virtual) representations  indexed by the subset $T,$ with Frobenius characteristic 
$$
f_n^T=\dfrac{1}{n}\sum_{d|n} \psi^T(d) p_d^{\frac{n}{d}}.$$
Set $F^T=\sum_{n\geq 1} f_n^T.$  Finally let 
 $p^T=\sum_{n\in T} p_n.$
\end{defn}

These definitions imply:
\begin{lem}\label{lem6.3} $f_n^T(1)=1$ if and only if $n\in T,$ and $f_n^T(1)=0$ otherwise.
\end{lem}
\begin{proof} Let us write $\delta(m\in T)$ for the indicator function of the set $T,$ so that $\delta(m\in T)=1$ if and only if $m\in T,$ and is zero otherwise.

By defnition of $\psi^T$, we have 
$$\psi^T(n)=\sum_{d|n}\mu(\tfrac{n}{d})\ d\,\delta(d\in T).$$
Hence M\"obius inversion gives 
$$n\delta(n\in T)=\sum_{d|n}\psi^T(d)
=n\, f_n^T(1),$$
i.e. $f_n^T(1)=\delta(n\in T)$ as claimed.
\end{proof}

With this lemma, we can now prove:
\begin{thm}\label{thm6.4} Let $T$ be a nonempty subset of the positive integers.  Then:
\begin{equation}\label{thm6.4eqn1} H[F^T]=\prod_{n\in T} (1-p_n)^{-1}
\end{equation}
\begin{equation}\label{thm6.4eqn2} F^T=p^T[Lie]=\sum_{m\in T} Lie[p_m],  \ \mathrm{or\ equivalently\ } 
f_n^T=\sum_{\stackrel{m\in T}{m|n}} Lie_{\frac{n}{m}}[p_m].
\end{equation}
If $G^T=\sum_{k\geq 0}\sum_{m\in T} Lie[p_{m\cdot 2^k}],$ then 
\begin{equation}\label{thm6.4eqn3}
E[G^T]=\prod_{n\in T} (1-p_n)^{-1}=H[F^T].
\end{equation}
\end{thm}

\begin{proof} 
Equation (\ref{thm6.4eqn1}) is immediate from Theorem \ref{metathm} and Lemma \ref{lem6.3}. 
 From Proposition \ref{Su1Prop3.1}, we have 
\begin{align*}\exp (F^T)&= \prod_{d\geq 1} (1-p_d)^{-\frac{1}{d}\psi^T(d)} =\prod_{d\geq 1} (1-p_d)^{-\sum_{m|d,\, m\in T} \frac{m}{d}\mu(\frac{d}{m})} \\
&=\prod_{d\geq 1} \prod_{m|d,\, m\in T} (1-p_d)^{- \frac{m}{d}\mu(\frac{d}{m})}
=\prod_{ m\in T}\prod_{r\geq 1}(1-p_{rm})^{- \frac{1}{r}\mu(r)}, \mathrm{\ putting\ }d=rm\\
&=\prod_{ m\in T}\prod_{r\geq 1}(1-p_{r})^{- \frac{1}{r}\mu(r)}[p_m]
=\prod_{ m\in T} \exp(Lie[p_m] )=\exp\sum_{m\in T} Lie[p_m],
\end{align*}
where we have used the first equation in Proposition \ref{Su1Prop3.1} for $F=Lie.$  Equation (\ref{thm6.4eqn2}) follows.
Equation (\ref{thm6.4eqn3})  is a consequence of Proposition \ref{Pleth3}.
\end{proof}

\begin{cor}\label{cor6.5} If either $F^T$ or $G^T$ is Schur-positive, then so is 
$$\prod_{n\in T} (1-p_n)^{-1}=\sum_{\stackrel{\lambda\in Par}{\lambda_i\in T}}p_\lambda.$$
\end{cor}

The following remarks will explain the connection between the functions $f_n^T$ and the work of Section 3.  
\begin{enumerate}
\item If $T=\{1\},$ then $\psi^T(d)=\mu(d)$ and $f_n^T$ corresponds to the representation $Lie_n.$
\item If $T$ is the set of all positive integers, then $\psi^T(d)=\phi(d)$ by M\"obius inversion of the well-known identity $m=\sum_{d|m}\phi(d).$  Thus $f_n^T$ is the characteristic of the conjugacy action
$\textbf{ 1}\uparrow_{C_n}^{S_n},$ i.e. $f_n^T=Conj_n.$
\item Fix a set $S$ of primes.  Let $T$ be the set of all integers whose prime factors are all in $S.$  Then clearly if $d\in T,$ $\psi^T(d)=\phi(d)$ by the identity used in (2). Otherwise $d=Q_d\ell_d$ with $Q_d\in T$ and $\ell_d$ relatively prime to $Q_d$ and also relatively prime to all integers in $T.$ Hence, since  $\mu$ is multiplicative,
Definition \ref{def6.1} gives 
$$\psi^T(d)=\sum_{m\in T} m \mu(Q_d/m) \mu(\ell_d)
=\mu(\ell_d)\cdot \sum_{m\in T} m \mu(Q_d/m)=\mu(\ell_d)\psi^T(Q_d),$$ 
 Since $Q_d\in T,$  we obtain  $\psi^T(d)=\mu(\ell_d)\phi(Q_d),$ which is precisely the formula given by (\ref{FoulkesPrimes1}). Thus $f_n^T=Lie_n^S.$
\item Fix a set $S$ of primes. Now let $T$ be the set of all integers relatively prime to every element of $S$ (so $T$ is the set of integers none of whose prime factors is in $S$). Exactly as above, we see that Definition \ref{def6.1} reduces to (\ref{FoulkesPrimes2}), and hence $f_n^T$ is the characteristic $Lie_n^{\bar S}.$
\end{enumerate}

From Theorem \ref{thm6.4} and the preceding observations, we have the following decompositions of the representations $Conj_n,$ $Lie_n^{(q)}.$ Only ~\eqref{prop6.6eqn4a}, recorded here for completeness, requires comment; it follows from ~\eqref{prop6.6eqn4} and Proposition~\ref{Pleth1} in Section~\ref{SecPleth}. 
\begin{prop}\label{prop6.6}
 \begin{equation}\label{prop6.6eqn1}\sum_{m\geq 1} p_m[Lie]=\sum_{n\geq 1} Conj_n;
\end{equation}
\begin{equation}\label{prop6.6eqn2}\sum_{m\geq 1} p_m=\sum_{n\geq 1}Conj_n[\sum_{r\geq 1}(-1)^{r-1} e_r]
\end{equation}
The plethystic inverse of $ Conj$ is 
\begin{equation}\label{prop6.6eqn3} (\sum_{n\geq 1} Conj_n)^{\langle -1\rangle}
= \sum_{r\geq 1}(-1)^{r-1} e_r[\sum_{n\geq 1} \mu(n) p_n].
\end{equation}
Let $q$ be prime, and let $n=\ell q^k $ 
where $(\ell, q)=1.$ Then 
\begin{equation}\label{prop6.6eqn4} Lie_n^{(q)} =\sum_{r=0}^k Lie_{\ell q^{k-r}}[p_{q^r}].\end{equation}
\begin{equation}\label{prop6.6eqn4a} Lie^{(q)} =\sum_{r\ge 0} Lie[p_{q^r}] \text{ and hence } Lie=Lie^{(q)}[p_1-p_q].
\end{equation}
The plethystic inverse of $Lie^{(q)}$ is 
\begin{equation}\label{prop6.6eqn5}\sum_{n\geq 1}(Lie_n^{(q)})^{\langle -1\rangle}
=Lie^{\langle -1\rangle}[p_1-p_q]=(\sum_{r\geq 1} (-1)^{r-1} e_r)[p_1-p_q]. 
\end{equation}

\end{prop}

The construct of Definition \ref{def6.2} allows us to remove the restriction that $q$ be prime, as follows.  Let $k\geq 2$ be any positive integer, and take $T$ to be the set of all nonnegative powers of $k.$ In this case Theorem \ref{thm6.4} gives
\begin{equation}\label{eqn6.9} H[\sum_{n\geq 1} f_n^T]=\prod_{r\geq 0} (1-p_{k^r})^{-1}, \qquad \sum_{n\geq 1} f_n^T=\sum_{r\geq 0} p_{k^r}[Lie].
\end{equation}
By inverting this equation plethystically, we obtain the recurrence 
\begin{equation}\label{eqn6.10} \text{For } k\geq 2, \qquad f_n^T=\begin{cases}  Lie_n+f_{\frac{n}{k}}[p_k], & k|n;\\
                          Lie_n, &\mathrm{otherwise},\\
\end{cases}
\end{equation}

However computations show that for $k=4,$ $f_n^T$ is not Schur-positive when $n=4, 16,$ and the degree 16 term in the product 
$\prod_{r\geq 0} (1-p_{4^r})^{-1}$ is not Schur-positive.  In both cases it is the sign representation that appears with coefficient $(-1).$

\begin{conj} For any ODD positive integer $k,$ $f_n^T$ as defined above is Schur-positive.  
\end{conj}

\begin{conj} The product $\prod_{r\geq 0} (1-p_{k^r})^{-1}$ is Schur-positive for any ODD positive integer $k.$
\end{conj}

Fix $k\geq 2$ and consider the subset $T=\{1,k\}.$ It was shown in \cite[Theorem 4.23]{Su1} that the symmetric function 
$$W_{n,k}=\sum_{\mu\vdash n, \mu_i=1\, or \, k} p_\mu$$ is Schur-positive.  Define $W_{0,k}=1.$ Then
$$\sum_{n\geq 0} W_{n,k}=\prod_{n\in T}(1-p_n)^{-1}=(1-p_1)^{-1} (1-p_k)^{-1}.$$
For $k=1$ we set $W_{n,1}=p_1^n$ for all $n\geq 0,$ so that the  preceding equation reduces, as expected, to
\begin{center} 
$\sum_{n\geq 0} W_{n,1}=\prod_{n\in T}(1-p_n)^{-1}=(1-p_1)^{-1}.$\end{center}

\begin{prop}\label{prop6.7} If $T=\{1,k\}$ and $k\geq 2,$ then 
\begin{equation}\label{eqn6.11}f_n^T
=\begin{cases}  Lie_n+Lie_{\frac{n}{k}}[p_k], & k|n;\\
                          Lie_n, &\mathrm{otherwise},\\
\end{cases}
\end{equation}
and hence $\sum_{n\geq 0} W_{n,k}=H[\sum_{n\geq 0} f_n^T].$

If $k$ is  prime, then 
$f_n^{\{1,k\}}=\mathrm{ch} (\exp\frac{2ki\pi}{n})\big\uparrow_{C_n}^{S_n}=\ell_n^{(k)},$ 
and hence the symmetric function defined by (\ref{eqn6.11}) is  
Schur-positive. \end{prop}
\begin{proof} Equation (\ref{eqn6.11}) is immediate from Theorem \ref{thm6.4}.  Now let $k$ be prime.   Recall Theorem \ref{Su1Thm5.6};  that equation  now becomes 
$$\sum_{\lambda\vdash n: \lambda_i=1, k} p_\lambda
=H[\sum_{m\geq 1} \text{ch}\, (\exp(\frac{2\pi i k}{n})\uparrow_{C_n}^{S_n})]\vert_{{\rm\, deg\,} n},$$ and thus the left-hand side 
is precisely $p^T$ for $k$ prime and $T=\{1,k\}.$ But $H-1$ is invertible with respect to plethysm, so  $H[F]=H[G]$ if and only if $F=G.$ Hence $f_n^T$ must coincide with $\ell_n^{(k)}$.
\end{proof}

Computations  indicate  that 
\begin{conj}  $f_n^{\{1,k\}}$ is Schur-positive for $k=2$  and for all odd $k\geq 3.$  (This is trivially true if $k=1.$)
\end{conj}

When $k$ is even and not equal to 2, this fails.  For instance, if $n=k=4m,$ it is easy to see that $Lie_{4m}+p_{4m}$ contains the sign representation with coefficient $(-1).$ However 
we have $H[F^{\{1,k\}}]=(1-p_1)^{-1}(1-p_k)^{-1},$ which we know to be Schur-positive from \cite[Proposition 4.23]{Su1}.  This example shows that it is not always possible to write a Schur-positive sum of power sums as a symmetrised module over a sequence of true $S_n$-modules, since $Lie_k+p_k$ fails to be Schur-positive when $k$ is even.

\begin{prop}\label{prop6.8} Let $k\geq 2.$ Then $\omega (E[F^{\{1,k\}}])=(1-p_1)^{-1}(1-(-1)^{k-1}p_k)^{-1} (1+p_2)(1+p_{2k}).$  If $k$ is prime, this is Schur-positive.
\end{prop}
\begin{proof}  We calculate $E[F^{\{1,k\}}]$ using Proposition \ref{Pleth5}.  Since the series $F^{\{1,k\}}$ is Schur-positive when $k$ is prime, the claim follows.\end{proof}

The next three propositions are also clear from Theorem \ref{thm6.4}.

\begin{prop}\label{prop6.9}  Let $k\geq 2$ and $T=\{n:n\leq k\}.$ Then 
$$\prod_{n=1}^k (1-p_n)^{-1}=H[\sum_{n} f_n^T]$$ 
for 
\begin{equation}\label{eqn6.12} f_n^T=\sum_{\stackrel{m=1}{m|n}}^k Lie_{\frac{n}{m}}[p_m].\end{equation}
\end{prop}

\begin{cor}\label{cor6.10}   Let $T=\{n:n\leq k\}, k\geq 2.$  If $n$ is prime, or $n\le k,$ or $n>k$ and $n$ is such that the greatest proper divisor of $n$ is at most $k,$ then $f_n^T$ is Schur-positive.
\end{cor}
\begin{proof} If $n$ is prime it is easy to see that 
$f_n^T=\begin{cases} Lie_n, & n>k\\
                                 Lie_n+p_n, & \mathrm{otherwise.}
\end{cases}$
Since in this case, $\mu(n)=-1$ and $\phi(n)=n-1,$ and thus $Lie_n=\frac{1}{n}(p_1^n-p_n),$ we conclude that 
$Lie_n+p_n=\frac{1}{n}(p_1^n +\phi(n)p_n) =Lie_n^{\mathcal{P}}.$

If $n\le k$ then the sum in equation (\ref{eqn6.12}) ranges over all divisors of $n$ and hence by (\ref{prop6.6eqn1}) we have $f_n^T=Conj_n.$ 

If $n>k$ and  the largest proper divisor of $n$ is at most $k,$ we have $f_n^T=Conj_n-p_n,$ again by (\ref{prop6.6eqn1}).  But it is well known that 
$p_n=\sum_{r\geq 0} (-1)^r s_{(n-r, 1^r)}.$  Also by a  result of 
\cite{Sw}, $Conj_n$ contains all hooks except for the following: $(n-1,1)$ for all $n\geq 2,$ 
$(2, 1^{n-2})$ for odd $n\geq 3$, and $(1^n)$ for even $n.$ In all three cases, the corresponding Schur function appears with coefficient $-1$ in $p_n,$ hence with coefficient  
$+1$ in $Conj_n-p_n.$ This finishes the proof. \end{proof}

\begin{conj} (See also \cite[Conjecture 1]{Su1}.) $f_n^{\{1,\ldots,k\}}$ is Schur positive for all $n$ and $k,$ and hence so is $\prod_{n=1}^k (1-p_n)^{-1}.$
\end{conj}

\begin{prop}\label{prop6.11} Let $k\geq 2$ and $T=\{n: n|k\}.$ Then 
\begin{equation}\label{eqn6.13} \prod_{n|k} (1-p_n)^{-1} = H[\sum_n f_n^T]
\end{equation} 
for \begin{equation}\label{eqn6.14} f_n^T=\sum_{m|(k,n)} Lie_{\frac{n}{m}}[p_m]
\end{equation}
\end{prop}

 In particular from Theorem \ref{Su1Thm5.6} we immediately have 
 (since $H-1$ is Schur-positive) that  
\begin{cor}\label{cor6.12} $$\ell_n^{(k)} ={\rm ch } \exp(2i\pi\cdot k/n)\uparrow_{C_n}^{S_n}= \sum_{m|(k,n)} Lie_{\frac{n}{m}}[p_m],$$
and hence $f_n^T$ is Schur-positive when $T$ is the set of all divisors of $k.$
\end{cor}
Hence:
\begin{cor}\label{cor6.13} We have the following decomposition of the regular representation into virtual representations:
\begin{equation}\label{eqn6.15} p_1^n=\sum_{k=1}^n \sum_{m|(k,n)} Lie_{\frac{n}{m}}[p_m]=\sum_{d|n}d\, Lie_d[p_{\frac{n}{d}}].
\end{equation}
\end{cor}
\begin{proof} This follows from the decomposition (see   \cite[Theorem 8.8]{R}) 
$$ p_1^n=\sum_{k=1}^n \ell_n^{(k)}$$
and the preceding corollary, because the first sum can be rewritten as $$\sum_{m|n} \sum_{\stackrel{r=1}{k=rm\leq n}}^{ \frac{m}{n}} Lie_{\frac{m}{n}}[p_m]=\sum_{m|n}\frac{m}{n}\ Lie_{\frac{m}{n}}[p_m].$$
\end{proof}

One can also arrive at the decomposition of equation (\ref{eqn6.15}) directly by using the expansion of $Lie_n$ into power sums.

\begin{prop}\label{prop6.14} Let $T=\{n:n\equiv 1\,\mathrm{mod}\, k\}.$
Then 
$$\prod_{n\equiv 1\,\mathrm{mod}\, k} (1-p_n)^{-1}=H[\sum_{n} f_n^T]$$ 
for $$f_n^T=\sum_{\stackrel{m\equiv 1\,\mathrm{mod}\, k}{ m|n}} Lie_{\frac{n}{m}}[p_m].$$
\end{prop}

After seeing a preprint  of \cite{Su1}, Richard Stanley made the following conjecture, and verified it for $n\leq 24$ and $k\leq 6$.
\begin{conj} (R. Stanley, 2015)  $\prod_{n\equiv 1\,\mathrm{mod}\, k} (1-p_n)^{-1}$ is Schur-positive for all $k.$ 
\end{conj}

As noted in equation (1) of Theorem \ref{Spositivity2} and other places, Conjecture 7 holds  for $k=2.$  We  have 
$$ \prod_{n\equiv 1\,\mathrm{mod}\, 2} (1-p_n)^{-1} =E[Conj]=H[L^{\overline{{(2)}}}].$$ In this case, writing $p^{\mathrm{ odd}}$ for 
$\sum_{n\, \mathrm{odd}} p_n, $ we have the identity
$$ p^{\mathrm {odd}}[Lie]=L^{\overline{{(2)}}},$$ and hence:

\begin{prop}\label{prop6.15}  $\sum_{\stackrel{m\,\mathrm{odd}}{ m|n}} Lie_{\frac{n}{m}}[p_m]=p^{\mathrm {odd}}[Lie]\vert_{{\rm deg\ } n}$ is  Schur-positive; it is the Frobenius characteristic  $Lie_n^{\overline{{(2)}}}$ of the representation $\exp(2i\pi\ell/n)\uparrow_{C_n}^{S_n},$ where 
$n=2^k\cdot\ell$ and $\ell$ is odd.
\end{prop}
\begin{proof} This is clear since the symmetric powers of the two modules coincide, both being equal to  $\prod_{n\, \mathrm{ odd}}(1-p_n)^{-1}.$
\end{proof}

Proposition \ref{prop6.6} gives us several different ways to decompose $Conj_n$, which we collect in the following:

\begin{thm}\label{thm6.16} 
 For any prime $q,$ we have 
\begin{equation}\label{eqn6.17} \sum_n Conj_n=\sum_{\stackrel{n}{ q \text{ does not divide } n}} p_n[Lie^{(q)}],
\end{equation}
and hence the sum on the right is Schur-positive.
In fact for any positive integer $q$ we have 
\begin{equation}\label{eqn6.18} \sum_n Conj_n=\sum_{\stackrel{n}{q \text{ does not divide } n}} p_n[\sum_{k\geq 0} Lie[p_{q^k}]].
\end{equation}

\end{thm}

\begin{proof} We start with equation (\ref{prop6.6eqn1}) of Proposition \ref{prop6.6}, and use the plethystic inverse formula established in Section 6,  Proposition \ref{Pleth1}.
We have, by associativity of plethysm,
\begin{align*} \sum_{n\geq 1} Conj_n&= \sum_{n\geq 1} p_n[Lie]\\
&=\sum_{n\geq 1} (p_n[p_1-p_q] )[\sum_{k\geq 0} p_{q^k}[Lie]]\\
&=\sum_{n\geq 1} (p_n-p_{nq} )[\sum_{k\geq 0} p_{q^k}[Lie]]\\
&=(\sum_{n\geq 1} p_n -\sum_{n\geq 1} p_{nq})[\sum_{k\geq 0} Lie[p_{q^k}]];
\end{align*}
invoking Theorem \ref{thm6.4}, this establishes equation (\ref{eqn6.18}).\end{proof}

\begin{rk}  In Section 2 it was conjectured that $Lie_n^{(2)}\uparrow_{S_n}^{S_{n+1}}-Lie_{n+1}^{(2)}$ is a true $S_{n+1}$ module which lifts $Lie_n^{(2)},$ if $n$ is not a power of 2.
One can ask if this holds for the odd primes $q.$

If $n$ is a power of an odd prime $q,$ it follows that 
$Lie_n^{(q)}={\rm ch\,} \textbf{ 1}\uparrow_{C_n}^{S_n}= Conj_n,$ while $Lie_{n-1}^{(q)}=Lie_n$ since $n-1\equiv q-1 \mod q,$ so that $n-1$ is relatively prime to $q.$ Now $Lie_n$ does not contain the sign representation (for $n\neq 2$) and never contains the trivial representation; however both appear once in the conjugacy action on the $n$-cycles, for odd $n.$ It follows that both these representations appear with multiplicity $(-1)$ in $p_1 Lie_{n-1}^{(q)}- Lie_n^{(q)}.$ 

For $q=3$ we have verified, up to $n=32,$ that $p_1 Lie_{n-1}^{(3)}-Lie_n^{(3)}$ is Schur-positive except for $n=3,6,9,10, 18, 27.$

For $q=5, n\leq 32,$ $p_1 Lie_{n-1}^{(5)}-Lie_n^{(5)}$ is Schur-positive except for $n=5,6,10, 25, 26.$
\end{rk}

\section{Meta theorems }\label{Secmetathms}
In this section we review the meta theorem of \cite{Su1} giving formulas for symmetric and exterior powers of modules induced from centralisers, 
and also further develop these tools in a general setting.  Some of the results of this section were announced in \cite{SuFPSAC2018}.

We begin by recalling the following results regarding the sequence of symmetric functions $f_n$ defined in equation (\ref{definef_n}) of Section 1:
\begin{prop}\label{Su1Prop3.1}\cite[Proposition 3.1]{Su1}  Define $F(t)=\sum_{i\geq 1} t^i f_i,$   and define 
$(\omega F)^{alt}(t)=\sum_{i\geq 1} (-1)^{i-1} t^i\omega(f_i).  $  Then 
\begin{align} &F(t)= \log \prod_{d\geq 1} (1-t^d p_d)^{-\frac{\psi(d)}{d}}\\ 
&(\omega F)^{alt}(t)= \log \prod_{d\geq 1} (1+t^d p_d)^{\frac{\psi(d)}{d}}\end{align}
\end{prop}

\begin{thm}\label{metathm} \cite[Theorem 3.2]{Su1} Let $F=\sum_{n\geq 1}  f_n$ where $f_n$ is of the form (\ref{definef_n}), $H(v)=\sum_{n\geq 0} v^n h_n$ and 
$E(v)=\sum_{n\geq 0}  v^n e_n.$  
We have the following plethystic generating functions:

\noindent
(Symmetric powers) 
\begin{equation}\label{metaSym}H(v)[F] = \sum_{\lambda\in Par} v^{\ell(\lambda)}  H_\lambda[F]=\prod_{m\geq 1} (1-p_m)^ {-f_m(v)}\end{equation}
(Exterior powers) 
\begin{equation}\label{metaExt}E(v)[F]=\sum_{\lambda\in Par} v^{\ell(\lambda)} E_\lambda[F] =\prod_{m\geq 1} (1- p_m)^{f_m(-v)}\end{equation}
 (Alternating exterior powers)
\begin{center}$\sum_{\lambda\in Par}  (-1)^{|\lambda|-\ell(\lambda)}v^{\ell(\lambda)} \omega(E_\lambda[F])$\end{center}
\begin{equation}\label{metaAltExt}=\sum_{\lambda \in Par} v^{\ell(\lambda)}H_\lambda[\omega(F)^{alt}]
=H(v)[\omega(F)^{alt}]
= \prod_{m\geq 1} (1+ p_m)^{f_m(v)}\end{equation}
(Alternating symmetric powers) 
\begin{center}$\sum_{\lambda\in Par} (-1)^{|\lambda|-\ell(\lambda)}v^{\ell(\lambda)} \omega(H_\lambda[F])$\end{center}
\begin{equation}\label{metaAltSym}=\sum_{\lambda \in Par} v^{\ell(\lambda)}E_\lambda[\omega(F)^{alt}]
=E(v)[\omega(F)^{alt}]
= \prod_{m\geq 1} (1+ p_m)^{-f_m(-v)}.\end{equation}
\end{thm}

\begin{prop}\label{metaf} \cite[Lemma 3.3]{Su1} The numbers $f_n(1)$ and $f_n(-1)$ determine each other according to the equations 
$f_{2m+1}(-1)=-f_{2m+1}(1)$ for all $m\geq 0,$ and 
$f_{2m}(-1)=f_m(1) -f_{2m}(1)$ for all $m\geq 1.$  In fact, 
the symmetric functions 
$f_n = \dfrac{1}{n} \sum_{d|n} \psi(d) p_d^{\frac{n}{d}}$
are  determined by the numbers 
 $f_n(1)=\dfrac{1}{n} \sum_{d|n}\psi(d) ,$
or by the numbers $f_n(-1)=\dfrac{1}{n} \sum_{d|n}\psi(d) (-1)^{\frac{n}{d}}.$
\end{prop}

Recall from Section 1.1 that we define  $H^{\pm}=\sum_{r\geq 0} (-1)^r h_r$ and $E^{\pm}=\sum_{r\geq 0} (-1)^r e_r.$ Thus 
$H^{\pm}=1-H^{alt},$ where $H^{alt}=\sum_{r\geq 1} (-1)^{r-1} h_r,$ and likewise $E^{\pm}=1-E^{alt}.$  The following identity is well known (see \cite[(2.6)]{M}, \cite[Section 7.6]{St4EC2}).
\begin{equation}\label{HE}\left(\sum_{n\geq 0} t^n h_n\right) \left( \sum_{n\geq 0} (-t)^n e_n\right) =1. \text{ Equivalently, } H^{\pm}\cdot E=1=H\cdot E^{\pm}.\end{equation}
This identity is generalised in Lemma \ref{pm} below.

\begin{lem}\label{pm} Let $F=\sum_{n\geq 1} f_n,$  $G=1+\sum_{n\geq 1} g_n$ and $K=1+\sum_{n\geq 1} k_n$ be arbitrary formal series of symmetric functions, as usual with $f_n, g_n, k_n$ being of homogeneous degree $n.$
\begin{enumerate}
\item 
$H[F]=G\iff E^{\pm}[F]=\dfrac{1}{G} \iff \sum_{r\geq 1} (-1)^{r-1} e_r[F] =\dfrac{G-1}{G}.$
\item $E[F]=K\iff H^{\pm}[F]=\dfrac{1}{K} \iff \sum_{r\geq 1} (-1)^{r-1} h_r[F] =\dfrac{K-1}{K}.$
\end{enumerate}
\end{lem}
\begin{proof}\begin{enumerate}
\item By definition,  $E^{\pm}=\sum_{r\geq 0} (-1)^r e_r=1/H,$ 
and hence the first equality follows. For the second equality, note that $\sum_{r\geq 1} (-1)^{r-1} e_r= 1-E^{\pm}=1 -1/H,$ and hence 
$\sum_{r\geq 1} (-1)^{r-1} e_r[F]= 1 -1/H[F]$ as claimed. The reverse direction is clear.
\item This follows exactly as above, since $H^{\pm}=\sum_{r\geq 0} (-1)^r h_r=1/E.$
\end{enumerate}
\end{proof}

\begin{lem}\label{pmalt} Let $G=\sum_{n\geq 1} g_n,$ $K=\sum_{n\geq 0} k_n,$ where $g_n, k_n$ are symmetric functions of homogeneous degree $n$ for $n\geq 1,$ and $k_0=1.$   Let $K^\pm$ denote the sum 
$\sum_{n\geq 0} (-1)^n k_n.$ Then 
\begin{align} 
H[\sum_{n\geq 1} (-1)^{n-1}\omega(g_n)]=K&\iff 
H[\sum_{n\geq 1} g_n] = \dfrac{1}{K[-p_1]}=\dfrac{1}{\omega(K)^\pm}\\
&\iff E^{\pm}[\sum_{n\geq 1} g_n] = \omega(K)^\pm.
\end{align}
\begin{align} 
E[\sum_{n\geq 1} (-1)^{n-1}\omega(g_n)]=K &\iff 
E[\sum_{n\geq 1} g_n] =  \dfrac{1}{K[-p_1]}=\dfrac{1}{\omega(K)^\pm}\\
&\iff H^{\pm}[\sum_{n\geq 1} g_n] = \omega(K)^\pm.
\end{align}

\end{lem}
\begin{proof} We use the fact that for any symmetric functions $f_1, f_2$ of homogeneous degree, $f_1[-f_2]=(-1)^{\mathrm{deg} f_1} \omega(f_1)[f_2].$ In particular this implies $K[-p_1]=\omega(K)^{\pm}$ whenever $K$ is a series of symmetric functions $k_n$ of homogeneous degree $n.$
Hence, using associativity of plethysm,
$$K=H[-\sum_{n\geq 1} (-1)^n  \omega(g_n)]=H[-G[-p_1]]
=(H[-G])[-p_1],$$
 or equivalently
$$K[-p_1]=E^{\pm}[G]=(\frac{1}{H})[G]=\dfrac{1}{H[G]},$$
and finally
$$H[G]=\dfrac{1}{K[-p_1]}=\dfrac{1}{\omega(K)^{\pm}}.$$
The equivalence of the first two equations  is a consequence of the fact that $H[G]=(\dfrac{1}{E^{\pm}})[G] =\dfrac{1}{E^{\pm}[G]}.$
The   equivalences of the second pair are obtained in a similar manner. 
\end{proof}

Lemma (\ref{pmalt}) explains, in greater generality,  the connection between equations (\ref{metaSym}) and (\ref{metaAltExt}) (resp. (\ref{metaExt}) and (\ref{metaAltSym})).  In fact these lemmas give us  the following observation. Let $F,$ $H,$ and $E$ be as defined in Theorem \ref{metathm}.  Then

\begin{prop}\label{metaequiv}  The identities of Theorem \ref{metathm} are all equivalent, and are also equivalent to 
\begin{equation}E^{\pm}(v)[F] 
=\prod_{m\geq 1} (1- p_m)^ {f_m(v)}\end{equation}

\begin{equation}H^{\pm}(v)[F]=\prod_{m\geq 1} (1- p_m)^{-f_m(-v)}\end{equation}
\end{prop}

Now let $F=\sum_{n\geq 1} f_n$ be an arbitrary series of symmetric functions $f_n$ homogeneous of degree $n.$   In particular $f_n$ need not be of the form (\ref{definef_n}). We write $F_{\geq 2}$ for the series $\sum_{n\geq 2} f_n.$  

\begin{thm}\label{metage2}\cite[Proposition 2.3, Corollary 2.4]{Su1} Assume that $F=\sum_{n\geq 1} f_n$ is any series of symmetric functions $f_n$ homogeneous of degree $n.$ Also assume $f_1=p_1.$  Then we have the following  identities:
 \begin{equation}\label{metage2a}
H(v)[F_{\geq 2}]
=E(-v) \cdot H(v)[F].
\end{equation}
Equivalently,
\begin{equation}\label{metage2b}
E^{\pm}(v)[F_{\geq 2}]=E^{\pm}(v)[\sum_{n\geq 2} f_n] 
=\dfrac{H(v)}{ H(v)[F]}
\end{equation}
\begin{equation}\label{metage2c}
E(v)[F_{\geq 2}]
=H(-v) \cdot E(v)[F].
\end{equation}
Equivalently,
\begin{equation}\label{metage2d}
H^{\pm}(v)[F_{\geq 2}]=H^{\pm}(v)[\sum_{n\geq 2} f_n] 
=\dfrac{E(v)}{ E(v)[F]}
\end{equation}

\end{thm}

\begin{proof} The equivalence of (\ref{metage2a}) and (\ref{metage2b}) follows because of the identity (\ref{HE}).  Consider now equation (\ref{metage2a}). By standard properties of the skewing operation and the plethysm operation (see, e.g. 
\cite[(8.8)]{M}), we know that 
$h_n[G_1+G_2]=\sum_{k=0}^n h_k[G_1] h_{n-k}[G_2].$
This in turn gives $$H[G_1+G_2]= H[G_1] \, H[G_2].$$  
Taking $G_1=f_1$ and $G_2=F- f_1,$ we have
$$ H[F] =H[f_1]\, H[\sum_{n\geq 2} f_n ].$$
But $H[ f_1]=H[ p_1]=\sum_{n\geq 0} h_n.$   
Hence, using (\ref{HE}), $$\dfrac{1}{H(v)[ f_1]}= \sum_{n\geq 0} (-v)^n e_n.$$
The equations (\ref{metage2c}) and (\ref{metage2d}) are obtained in  entirely analogous fashion.
\end{proof}

An important consequence of Theorem \ref{metage2} is worth pointing out.  Denote by $Lie_n$ the Frobenius characteristic of the $S_n$-representation afforded by the multilinear component of the free Lie algebra on $n$ generators.  Let $Lie=\sum_{n\geq 1} Lie_n.$ A old theorem of Thrall \cite{T}, equivalent by Schur-Weyl duality to the Poincar\'e-Birkhoff-Witt theorem (see \cite{R} for background) states that 
\begin{equation*} H[Lie]=(1-p_1)^{-1} \end{equation*}
This identity is the special case of equation (\ref{metaSym}),  Theorem \ref{metathm}, obtained by taking, in \ref{definef_n}, $\psi(d)=\mu(d)$ and hence $f_n=\text{ch }(Lie_n).$  See 
Theorem \ref{ThrallPBWCadoganSolomon} and more generally\cite{R}. 
Define a symmetric function $\kappa=\sum_{n\geq 2} s_{(n-1,1)},$ where $s_{(n-1,1)}$ is the Schur function indexed by the partition $(n-1,1).$ (This is the Frobenius characteristic of the standard  representation of $S_n.$) 

Lemma \ref{pm} and Theorem \ref{metage2} now imply that the theorems of Thrall and Poincar\'e-Birkhoff-Witt are in fact {\it equivalent} to the derived series decomposition of the free Lie algebra \cite{R}.  More precisely, the following identities are equivalent:
\begin{thm}\label{EquivPBW}  (Equivalence of Poincar\'e-Birkhoff-Witt and derived series for free Lie algebra)
\begin{equation}\label{EquivPBWSym}  (H-1)[Lie]=(\sum_{r\geq 1} h_r)[Lie]=\sum_{n\geq 1} p_1^n. \end{equation}
\begin{equation}\label{PlInv-HCadogan}  (H-1)[\sum_{i\geq 1} (-1)^{i-1}\omega(Lie_i)]=p_1,
\end{equation}
(the plethystic inverse of the sum of homogeneous symmetric functions $\sum_{r\geq 1} h_r$).
\begin{equation}\label{EquivPBWAltExt}(1-E^{\pm})[Lie]= (\sum_{r\geq 1} (-1)^{r-1} e_r)[Lie]=p_1,\end{equation}
(the plethystic inverse of the sum $\sum_{r\geq 1} Lie_r$).
\begin{equation}\label{EquivPBWAltExtge2a}(1-E^{\pm})[Lie_{\geq 2}] =(\sum_{r\geq 1} (-1)^{r-1} e_r)[Lie_{\geq 2}]=\kappa \end{equation}
Equivalently,
\begin{equation}\label{EquivPBWAltExtge2b}  \text{ The degree $n$ term in }\sum_{r\geq 0} (-1)^{n-r} e_{n-r}[Lie_{\geq 2}] \text{ is }(-1)^{n-1}  s_{(n-1,1)},
\end{equation}
\begin{equation}\label{LieFilta}   Lie_{\geq 2}=Lie[\kappa]\end{equation}
\begin{equation} \label{LieFiltb}
 Lie_{\geq 2}=\kappa+\kappa[\kappa]+\kappa[\kappa[\kappa]]+\ldots   \text{ (The derived series filtration of the free Lie algebra)  } \end{equation}
 \begin{equation}\label{LieHodge} (H-1)[Lie_{\geq 2}]=(\sum_{r\geq 1} h_r)[Lie_{\geq 2}]=(1-p_1)^{-1}\cdot E^{\pm}-1
 =\sum_{n\geq 2}\sum_{k=0}^n (-1)^k p_1^{n-k}e_k. \end{equation}
\end{thm}

\begin{proof} We specialise the preceding identities to $v=1.$  Equation (\ref{EquivPBWSym}) is equivalent to $H[Lie]=(1-p_1)^{-1},$ and hence by Lemma \ref{pm}, we obtain $E^{\pm}[Lie]=1-p_1,$ which is equation (\ref{EquivPBWAltExt}) after removing the constant term and adjusting signs.

Now invoke (\ref{metage2b}) of Theorem \ref{metage2}.  We have 
$$E^{\pm}[Lie_{\geq 2} ]=H\cdot (1-p_1)
=1+ \sum_{n\geq 2} (h_n-h_{n-1} p_1)
= 1-\kappa, $$
and this is precisely equation (\ref{EquivPBWAltExtge2a}), again after cancelling the constant term and adjusting signs.  Since these steps are clearly reversible, we see that (\ref{EquivPBWAltExt}) and (\ref{EquivPBWAltExtge2a}) are in fact equivalent.

The equivalence of (\ref{EquivPBWAltExtge2a}) and (\ref{EquivPBWAltExtge2b})-(\ref{LieFilta}) follows by applying the plethystic inverse of $\sum_{r\geq 1} (-1)^{r-1}e_r, $ which is given by (\ref{EquivPBWAltExt}). 

It is clear by iteration that (\ref{LieFilta}) gives (\ref{LieFiltb}).  In the reverse direction, we can rewrite (\ref{LieFiltb}) as 
\begin{center}
$Lie=p_1+\kappa+\kappa[\kappa]+\kappa[\kappa[\kappa]]+\ldots,$\end{center} 
and hence 
$Lie[\kappa]=Lie-p_1=Lie_{\geq 2},$
which is (\ref{LieFilta}).
Finally the equivalence of (\ref{EquivPBWSym}) and ((\ref{LieHodge})) follows from equation (\ref{metage2a}) of Theorem \ref{metage2} and equation (\ref{HE}).  
\end{proof} 

\begin{thm}\label{Restrictge2}  Let $F,H,E$ be as in Theorem \ref{metathm}, and assume $f_1=p_1.$ 
\begin{enumerate}
\item Let $\prod_{m\geq 2} (1-p_m)^{-f_m(1)}=\sum_{n\geq 0} g_n$ for homogeneous symmetric functions $g_n$ of degree $n,$ $g_0=1.$  (Note that $g_1=0.$)
Also define $\sigma_n=\sum_{i\geq 0} (-1)^i e_{n-i} g_i.$ Then $\sigma_n, n\geq 1,$ is the characteristic of a one-dimensional, possibly  virtual representation, with the property that its restriction to $S_{n-1}$ is $\sigma_{n-1}.$   Let $\alpha_n$ be the degree $n$ term in $H[F_{\geq 2}], n\geq 0.$ (Note that $\alpha_0=1$ and $\alpha_1=0.$ )  Then we have the recurrence
\begin{equation} 
\alpha_n=p_1\alpha_{n-1}+(-1)^n\sigma_n.
\end{equation}
\item  Let $\prod_{m\geq 2} (1-p_m)^{f_m(-1)}=\sum_{n\geq 0} k_n$ for homogeneous symmetric functions $k_n$ of degree $n,$ $k_0=1.$  (Note that $k_1=0.$)
Also define $\tau_n=\sum_{i\geq 0} (-1)^i h_{n-i} k_i.$ Then $\tau_n, n\geq 1, $ is the characteristic of a one-dimensional, possibly  virtual representation, with the property that its restriction to $S_{n-1}$ is $\tau_{n-1}.$ Let $\beta_n$ be the degree $n$ term in $E[F_{\geq 2}], n\geq 0.$ (Note that $\beta_0=1$ and $\beta_1=0.$ )  Then we have the recurrence
\begin{equation} 
\beta_n=p_1\beta_{n-1}+(-1)^n\tau_n.
\end{equation}
\end{enumerate}
\end{thm}

\begin{proof}  \begin{enumerate} 
\item The hypothesis that the degree one term $f_1$ in $F$ equals $p_1$ implies that $f_1(1)=1.$ 
 From equation (\ref{metaSym}) of Theorem \ref{metathm} and (\ref{metage2a}) of Theorem \ref{metage2}, we now have 
\begin{align*}
(1-p_1) H[F_{\geq 2}]&= E^{\pm} \cdot \prod_{m\geq 2} (1-p_m)^{-f_m(1)}\\
&=\sum_{r\geq 0} (-1)^r e_r\sum_{n\geq 0} g_n=\sum_{n\geq 0} \sum_{i=0}^n (-1)^{n-i} e_{n-i} g_i\\
&=\sum_{n\geq 0} (-1)^n \sigma_n,
\end{align*}
from which the recurrence is clear.  It remains to establish the statement about $\tau_n.$ First observe that since $p_1$ does not appear in the power-sum expansion of $g_n,$ for $i\geq 1,$ $e_{n-i}g_i$ is the Frobenius characteristic of a zero-dimensional (hence virtual) representation (dimension is computed, for example,  by taking the scalar product with $p_1^n$).  The dimension of $\sigma_n$ is therefore that of $e_n,$ and is thus one.  To verify the statement about the restriction, we use the fact that the Frobenius characteristic of the restriction is the partial derivative $\dfrac{\partial \sigma}{\partial p_1}.$ The partial derivative of $e_n$ with respect to $p_1$ is clearly $e_{n-1}, n\geq 1,$ and that of $g_n$ with respect to $p_1$ is clearly 0.  The claim follows.

\item Again, $f_1=p_1$ implies $1=-f_1(-1).$ The argument is identical, but now use equation (\ref{metaExt}) and equation (\ref{metage2c}).
\end{enumerate}
\end{proof}

Note that, with $F$ as in Theorem \ref{metathm},  the dimension of the representation whose characteristic is $h_j[F]|_{{\rm deg\ }n} $ (respectively $e_j[F]|_{{\rm deg\ }n} $) is the number $c(n,j)$ of permutations in $S_n$ with $j$ disjoint cycles. 
Similarly  the dimension of the representation whose characteristic is $h_j[F_{\geq 2}]|_{{\rm deg\ }n} $ 
(respectively $e_j[F_{\geq 2}]|_{{\rm deg\ }n} $) is the number $d(n,j)$ of fixed-point-free permutations, or derangements, in $S_n$ with $j$ disjoint cycles, and hence the dimension of $\alpha_n $ (respectively $\beta_n$) is the total number of derangements $d_n.$ Hence, taking dimensions in either of the above recurrences, we recover the well-known recurrence $d_n=nd_{n-1} +(-1)^n, n\geq 2.$   On the other hand, the recurrence (\ref{Restrict2}) below is the symmetric function analogue of the recurrence 
$$d(n,j)=n(d(n-1,j)+d(n-2, j-1)),$$ while (\ref{Restrict1}) is the analogue of 
$$c(n,j)=c(n-1,j)+n c(n-1, j-1).$$

From Theorem \ref{metage2} we can also deduce interesting recurrences for the restrictions of the symmetric and exterior powers of $F$ from $S_n$ to $S_{n-1},$ for an arbitrary formal sum $F$ of homogeneous symmetric functions $f_n$ having the following key property: each $f_n$ is the Frobenius characteristic of an $S_n$-representation (possibly virtual) which restricts to the regular representation of $S_{n-1}.$ See also \cite[Proposition 3.5]{Su0}.  

\begin{thm}\label{Restrict}  Let $F=\sum_{n\geq 1} f_n.$  Assume $f_n$ is a symmetric function of homogeneous degree $n$ with the following property: $\dfrac{\partial}{\partial p_1} f_n=p_1^{n-1}, n\geq 1.$ 
\begin{enumerate}
\item Let ${G}^j_n$ equal either $h_j[F]|_{{\rm deg\ }n}$  or  $e_j[F]|_{{\rm deg\ }n}.$ Then for $n\geq 1$ and $0\leq j\leq n$ we have
\begin{equation} \label{Restrict1}
\dfrac{\partial}{\partial p_1} {G}^{n-j}_{n}={G}^{n-1-j}_{n-1}+p_1\dfrac{\partial}{\partial p_1} {G}^{n-1-(j-1)}_{n-1} 
\end{equation}
\item 
 Let $\hat{G}^j_n$ equal either $h_v[F_{\geq 2}]|_{{\rm deg\ }n}$   or $e_j[F_{\geq 2}]|_{{\rm deg\ }n}.$ Then for $n\geq 2$ and $1\leq j\leq n-1,$  we have
\begin{equation}\label{Restrict2}
\dfrac{\partial}{\partial p_1} \hat{G}^{n-j}_n=p_1(\dfrac{\partial}{\partial p_1} \hat{G}^{(n-1)-(j-1)}_{n-1} +\hat{G}^{n-2-(j-1)}_{n-2})
\end{equation}
\end{enumerate}
\end{thm}

\begin{proof}  The hypothesis about the $f_n$ implies that  derivative of $F$ with respect to $p_1$ is $\sum_{n\geq 1} p_1^{n-1}=(1-p_1)^{-1}.$ Also note that 
$$\dfrac{\partial}{\partial p_1}H(v)=vH(v), \qquad
\dfrac{\partial}{\partial p_1}E(v)=vE(v).$$
\begin{enumerate}
\item  The chain rule gives $$\dfrac{\partial}{\partial p_1}\left(H(v)[F]\right)=v\cdot H(v)[F]\cdot (1-p_1)^{-1},$$ 
and hence 
\begin{equation*}v\cdot H(v)[F]=(1-p_1)\dfrac{\partial}{\partial p_1}\left(H(v)[F]\right).\qquad (A)\end{equation*}
Similarly \begin{equation*}v\cdot E(v)[F]=(1-p_1)\dfrac{\partial}{\partial p_1}\left(E(v)[F]\right).\qquad(B) \end{equation*}
If ${G}^j_n=h_j[F]|_{{\rm deg\ }n}, $ then 
$H(v)[F]=\sum_{j\geq 0}\sum_{n\geq 0} G^j_n.$
The result follows by extracting the symmetric function of degree $n-1$ on each side of $(A)$, and the coefficient of $v^{n-j}.$ 

The  recurrence for $e_j[F]$ is identical in view of equation $(B)$; now use the expansion $E(v)[F]=\sum_{j\geq 0}\sum_{n\geq 0} G^j_n.$

\item  Now we use the identity (\ref{metage2a}) of Theorem \ref{metage2}.  We have 
$$H(v)[F_{\geq 2}]=E^{\pm}(v)\cdot H(v)[F].$$
Using the fact that $$\dfrac{\partial}{\partial p_1}E^{\pm}(v)=-v\cdot E^{\pm}(v),$$ and the chain rule, we obtain
\begin{align*} 
\dfrac{\partial}{\partial p_1} (H(v)[F_{\geq 2}]) &=-v\cdot E^{\pm}(v) \cdot H(v)[F_{\geq 2}]
+E^{\pm}\cdot \dfrac{\partial}{\partial p_1} (H(v)[F_{\geq 2}])\\
&= -v\cdot H(v)[F_{\geq 2}] +E^{\pm}\cdot v\cdot H(v)[F]\cdot (1-p_1)^{-1},
\end{align*}
where we have used the computation in (1).  It follows that 
$$(1-p_1)\dfrac{\partial}{\partial p_1} (H(v)[F_{\geq 2}])
= v\cdot p_1 (H(v)[F_{\geq 2}]),$$
and hence
\begin{equation*}\dfrac{\partial}{\partial p_1} (H(v)[F_{\geq 2}])
=p_1\left( \dfrac{\partial}{\partial p_1} (H(v)[F_{\geq 2}])+
v\cdot \dfrac{\partial}{\partial p_1} (H(v)[F_{\geq 2}])\right).\end{equation*}
 Let $H(v)[F_{\geq 2}]= \sum_{i\geq 0} \sum_{n\geq 1} v^j \hat{G}^j_n.$
The recurrence follows by extracting the symmetric function of degree $n-1$ on each side, and the coefficient of $v^{n-j}.$  
Similarly, in view of the identity (\ref{metage2c}) of Theorem \ref{metage2} and the fact that $$\dfrac{\partial}{\partial p_1}H^{\pm}(v)=-v\cdot H^{\pm}(v),$$
we obtain 
\begin{equation*}\dfrac{\partial}{\partial p_1} (E(v)[F_{\geq 2}])
=p_1\left( \dfrac{\partial}{\partial p_1} (E(v)[F_{\geq 2}])+
v\cdot \dfrac{\partial}{\partial p_1} (E(v)[F_{\geq 2}])\right).\end{equation*}
Hence it is clear that the same recurrence holds for $e_j[F_{\geq 2}].$
\end{enumerate}
\end{proof}

In particular, Theorem \ref{Restrict} applies to the family of representations whose characteristic $f_n$ is defined by equation (\ref{definef_n}), provided $\psi(1)=1.$ The latter condition guarantees that each $f_n$ restricts to the regular representation of $S_{n-1}.$ 

In the recent paper \cite{HR}, Hersh and Reiner derive several identities and recurrences for what are essentially the symmetric and exterior powers of $Lie.$  The connection between the work of \cite{HR} and the specialisation of our results to $F=Lie$, is the well-known fact ( see \cite{LS}, \cite{Su1}) that the exterior power of $Lie,$ when tensored with the sign representation, coincides with the Whitney homology of the lattice of set partitions.  Here we record the conclusions for the special setting of $F=Lie$.   In the notation of \cite{HR}, we have ${\rm ch\ }\widehat{Lie}^i_n=h_{n-i}[Lie_{\geq 2}]\vert_{{\rm deg\ }n},$ 
while ${\rm ch\ }\omega(\widehat{W}^i_n)=e_{n-i}[Lie_{\geq 2}]\vert_{{\rm deg\ }n}.$  Also let $\widehat{Lie}_n=\sum_{i\geq 0}^{n-1}\widehat{Lie}^i_n,$ and $\widehat{W}_n=\sum_{i\geq 0}^{n-1}\widehat{W}^i_n.$

\begin{cor}\label{HRRepStab} (The case $F=Lie$)
\begin{enumerate}
\item (\cite[Theorem 1.7]{HR})  
$$\sum_{i\geq 0}(-1)^i {\rm ch\ }\omega(\widehat{W}^i_n)
=(-1)^{n-1} s_{(2, 1^{n-2})}.$$
\item (\cite[Theorem 1.2]{HR})  
 $$ {\rm ch\ }\widehat{Lie}_n=
(H-1)[Lie_{\geq 2}]|_{{\rm deg\ }n}= p_1 \cdot  {\rm ch\ }\widehat{Lie}_{n-1}+(-1)^n e_n 
$$
and 
$$ {\rm ch\ }\widehat{W}_n=
(E-1)[Lie_{\geq 2}]|_{{\rm deg\ }n}= p_1 \cdot {\rm ch\ }\widehat{W}_{n-1} +(-1)^n \tau_n , $$
where $\tau_n= s_{(n-2, 1^2)}-s_{(n-2,2)},n \geq 4,$ and $\tau_3=s_{(n-2, 1^2)}.$
 
 \item  \cite[Theorem 1.4]{HR} $$\dfrac{\partial}{\partial p_1} {\rm ch\ }\widehat{Lie}^j_n=p_1(\dfrac{\partial}{\partial p_1} {\rm ch\ }\widehat{Lie}^{(j-1)}_{n-1} +{\rm ch\ }\widehat{Lie}^{(j-1)}_{n-2})$$
 $$\dfrac{\partial}{\partial p_1} {\rm ch\ }\widehat{W}^j_n=p_1(\dfrac{\partial}{\partial p_1} {\rm ch\ }\widehat{W}^{(j-1)}_{n-1} +{\rm ch\ }\widehat{W}^{(j-1)}_{n-2})$$
\end{enumerate}
\end{cor}
\begin{proof} Clearly (1) is just equation (\ref{EquivPBWAltExtge2b}) tensored with the sign.

For (2),  apply Theorem \ref{Restrictge2}  to $F=Lie=\sum_{n\geq 1} Lie_n.$  
It is clear that in this case we have $g_n=0, n\geq 2,$ and $k_2=-p_2, k_n=0, n\geq 3.$ Hence Theorem \ref{Restrictge2} gives the following:
  \begin{equation*}
(H-1)[Lie_{\geq 2}]|_{{\rm deg\ }n}= p_1 \cdot (H-1)[Lie_{\geq 2}]|_{{\rm deg\ }n-1} +(-1)^n e_n 
\end{equation*}
and 
\begin{equation*}
(E-1)[Lie_{\geq 2}]|_{{\rm deg\ }n}= p_1 \cdot (E-1)[Lie_{\geq 2}]|_{{\rm deg\ }n-1} +(-1)^n \tau_n , 
\end{equation*}
where $\tau_n=h_n-h_{n-2} p_2 = s_{(n-2, 1^2)}-s_{(n-2,2)},n \geq 4,$ and $\tau_3=s_{(n-2, 1^2)}.$

Part (3) is immediate from Theorem \ref{Restrict}, which applies since it is well known that $Lie_n$ restricts to the regular representation of $S_{n-1}.$  When $F=Lie_n,$ the functions $G_n^j$ become ${\rm ch\ }\widehat{Lie}_n^{n-j}$ when applied to the symmetric powers $H$, and $\omega({\rm ch\ }\widehat{W}_n^{n-j})$  when applied to the exterior powers $E$.
\end{proof}

  Hersh and Reiner use the second recurrence in (2) above  to establish a remarkable formula for the decomposition into irreducibles for  $(E-1)[Lie_{\geq 2}]|_{{\rm deg\ }n},$ in terms of certain standard Young tableaux that they call {\it Whitney-generating tableaux} \cite[Theorem 1.3]{HR}.   The decomposition into irreducibles of $(H-1)[Lie_{\geq 2}]|_{{\rm deg\ }n}$ is similarly given as a sum of {\it desarrangement  tableaux} \cite[Section 7]{HR}. This was established from the first recurrence in (2) above, for the sign-tensored version, in \cite[Proposition 2.3]{RW}. 

   In previous sections of this paper we  saw how these general theorems, especially Theorems \ref{metathm} and \ref{metage2}-\ref{Restrict},  can be applied to other families of representations, yielding analogues of all the results for $Lie_n.$


\section{A compendium of plethystic inverses}\label{SecPleth}
This section contains some useful plethystic identities and manipulations.



\begin{prop}\label{Pleth1} We have 
\begin{enumerate}
\item For $q\geq 2,$ $p_1-p_q$ and $\sum_{k\geq 0}p_{q^k}$ are plethystic inverses. 
\item For $q\geq 2,$  $p_1+p_q$ and $\sum_{k\geq 0}(-1)^k p_{q^k}$ are plethystic inverses. 
\item For $q\ge 2, H[p_1-p_q]=\dfrac{H}{H[p_q]}.$
\item $\dfrac{H}{H[p_2]}=H[p_1-p_2]=E,$  and hence $(H-1)[p_1-p_2]=E-1.$
\item Suppose $F=\sum_{n\geq 1} f_n$ and $G=1+\sum_{n\geq 1} g_n$ are formal series of symmetric functions,  with $f_n, g_n$ being of homogeneous degree $n.$ Then $H[F]=G\iff G
=(\tfrac{H}{H[p_q]})[\sum_{k\ge 0} F[p_{q^k}]].$  In particular 
\[H[F]=G\iff  G=E[\sum_{k\ge 0} F[p_{2^k}]].\]
\end{enumerate}
\end{prop}

\begin{proof} Part (1) follows by calculating 
$$(p_1-p_q)[\sum_{k\geq 0}p_{q^k}]=\sum_{k\geq 0}p_{q^k}-
\sum_{k\geq 0}p_q[p_{q^k}]
=\sum_{k\geq 0}p_{q^k}-\sum_{k\geq 0}p_{q^{k+1}}=p_1.$$
Part (2) follows in the same manner.
For Part (3), use the exponential generating function (see \cite{M})
$H=\exp\sum_{i\geq 0} \frac{p_i}{i}.$  Then 
$$H[p_1-p_q]=\exp\sum_{i\geq 0} \frac{p_i}{i}[p_1-p_q]
=\exp\sum_{i\geq 0} \tfrac{1}{i}(p_1-p_q)[{p_i}]
=\frac{\exp\sum_{i\geq 0} \frac{p_i}{i}}{\exp\sum_{i\geq 0}\frac{p_{qi}}{i}}=\frac{H}{H[p_q]}.$$
Now Part (4) follows by observing that 
$\sum_{i\geq 0} \frac{p_i}{i}-\sum_{i\geq 0}\frac{p_{2i}}{i}
=\sum_{i\geq 0} (-1)^{i-1}\frac{p_i}{i}, $ and hence $H[p_1-p_2]=E,$ using the exponential generating function for $E.$

 
 For Part (5), use Part (1) to rewrite  
 $H[F]=(H[p_1-p_q][\sum_{k\geq 0}p_{q^k}])[F],$ and then use Part (3) and the fact that plethysm is associative.
 \end{proof}

\begin{prop}\label{Pleth3} Let $F$ and $G$ be two series of symmetric functions with no constant term.  Then  the following are equivalent:
\[H[F]=E[G]\iff E^{\pm}[F]=H^{\pm}[G]\iff F=G-G[p_2]\iff G=\sum_{k\geq 0} F[p_{2^k}].\]
In particular, any  symmetric power of modules is also expressible as an exterior power of modules.  Also, if $F$ is Schur positive, then so is $G-G[p_2].$
\end{prop}
\begin{proof}  The equivalence of the first two statements follows  from Lemma \ref{pm}.  We have $(H-1)[F]=(E-1)[G],$ and hence, using associativity of plethysm, 
$F=((H-1)^{\langle-1\rangle}[E-1])[G],$
$G=((E-1)^{\langle-1\rangle}[H-1])[F].$  

Using Proposition~\ref{Pleth1}, we have 
$H[F]=E[G]=(H[p_1-p_2])[G]$ and hence $F= (p_1-p_2)[G]=G-G[p_2].$

Similarly, using the plethystic inverse of $p_1-p_2,$ we have 
$H[F]=H[p_1[F]]=(H[p_1-p_2])([\sum_{k\geq 0} p_{2^k}[F]])
=E[\sum_{k\geq 0} p_{2^k}[F]]=E[\sum_{k\geq 0} F[p_{2^k}]],$ and hence $G=\sum_{k\geq 0} F[p_{2^k}].$  In both cases we use the fact that $H-1, E-1$ are invertible with respect to plethysm.
\end{proof}

Equations $(A)$ and $(B)$ in the following example illustrate  Proposition \ref{Pleth3} and Theorem \ref{Spositivity2}.
\begin{ex}\label{DualityEx}  Let $S=\{2\}.$ 
\begin{equation*}
(1-p_1)^{-1}=H[Lie^\emptyset]=H[Lie]=E[Lie^{S}]=E[Lie^{(2)}],\qquad (A)
\end{equation*}
 from equations  (\ref{SymLS}) and Part (1) of Theorem \ref{PlInv-E}.
 
 Proposition \ref{Pleth3} gives $Lie=Lie^{(2)}-Lie^{(2)}[p_2].$

%

\vskip .05in
Now let $S$ be the set of odd primes. Then
\begin{equation*}
\prod_{ n {\rm \, odd}} (1-p_n)^{-1}=H[Lie^{\overline{\{2\}}}]=E[Lie^{\bar{\emptyset}}]
=E[Conj], \qquad (B)
\end{equation*}
 from Theorem \ref{SymLS} and Proposition \ref{ExtLieConj}. 
In this case, the preceding proposition asserts that  $Lie^{\overline{\{2\}}}=Conj-Conj[p_2];$ in particular the latter is Schur positive.
 

\end{ex}

For arbitrary primes $q$ we have:

\begin{prop}\label{Pleth4} Let $g_n=\dfrac{1}{n}\sum_{d|n} \psi(d) p_d^{\frac{n}{d}},$ and let $G=\sum_{n\geq 1} g_n.$ 

Let $F=(p_1\pm p_q)[G]=\sum_{n\geq 1} f_n.$ 

Then
$f_n=\dfrac{1}{n}\sum_{d|n} \bar{\psi}(d) p_d^{\frac{n}{d}},$
where 
$\bar\psi(d)=
\begin{cases} \psi(d)\pm q\psi(\frac{d}{q}), &q|d,\\
                    \psi(d), & \mathrm{otherwise}.
\end{cases}$
\end{prop}
\begin{proof} We will do the case $F=(p_1-p_q)[G],$ since the other case is identical. Note that $F=(p_1-p_q)[G]= G-p_q[G]=G-G[p_q].$ 
It is easiest to use Proposition \ref{Su1Prop3.1}:
 \begin{align*}
F &=\log \prod_{d\geq 1} (1- p_d)^{-\frac{\psi(d)}{d}}
-\log \prod_{d\geq 1} (1- p_{dq})^{-\frac{\psi(d)}{d}}\\
&=\log \dfrac{\prod_{d\geq 1} (1- p_d)^{-\frac{\psi(d)}{d}}}
{\prod_{d\geq 1} (1- p_{dq})^{-\frac{\psi(d)}{d}}}
=\log \left(\prod_{d\geq 1,\, q \not{|} d} (1- p_d)^{-\frac{\psi(d)}{d}} \prod_{m\geq 1} (1- p_{mq})^{-\frac{\psi(mq)}{mq}+\frac{\psi(m)}{m}}\right)
\end{align*}
\end{proof}

We collect the plethystic inverses, from this paper and from other authors, that are of homological and representation-theoretic significance, and indeed, were first derived in that context. Let $g$ be a symmetric function without constant term, and with nonzero term of degree 1. Denote by $g^{\langle-1\rangle} $ the plethystic inverse of $g. $ Equations (\ref{Pleth8a}) and (\ref{Pleth8b}) first appeared in \cite{CHR} and \cite{CHS} respectively.   
\begin{prop}\label{Pleth8}   The following pairs are plethystic inverses:
\begin{equation}\label{Pleth8a}
 \dfrac{p_1}{1+p_1} {\rm\ and\ } \dfrac{p_1}{1-p_1} ;\end{equation}
\begin{equation}\label{Pleth8b}\sum_{n\geq 1} g(n) p_n {\rm\ and\ }\sum_{n \geq 1}g(n) \mu(n) p_n, \end{equation}
for any function $g(n)$ defined on the positive integers, such that $g(mn)=g(m)g(n);$
\begin{equation}\label{Pleth8c}\sum_{n\text{ odd}} g(n) p_n {\rm\ and\ }\sum_{n\text{ odd}} g(n) \mu(n) p_n, \end{equation}
for any function $g(n)$ defined on the positive integers, such that $g(mn)=g(m)g(n);$
 \begin{equation}\label{Cadogan} \sum_{i\geq 1} (-1)^{i-1} \omega(Lie_i) \text{ and } H-1; \end{equation}  
 \begin{equation}\label{Sheila} \sum_{i\geq 1} (-1)^{i-1} \omega(Lie_i^{(2)})
 \text{ and } E-1; \end{equation}
\begin{equation} \label{Pleth8d}
\sum_{n\geq 1} Lie_n  {\rm\ and\ }\dfrac{H-1}{H}=\sum_{n\geq 1} (-1)^{n-1} e_n;
\end{equation}
\begin{equation} \label{Pleth8e}
\sum_{n\geq 1} Lie_n^{(2)} {\rm\ and\ } \dfrac{E-1}{E}=\sum_{n\geq 1} (-1)^{n-1} h_n,
\end{equation}
and hence $\omega(Lie_n^{\langle-1\rangle})
=(Lie_n^{(2)})^{\langle-1\rangle}.$

\begin{equation}\label{Pleth8f}
\sum_{n\equiv 1 {\rm\, mod\, }k} h_n {\rm\ and\ }\sum_{n\geq 0} (-1)^n \beta_{nk+1}; 
\end{equation}
here $\beta_{nk+1} $ is the Frobenius characteristic of the homology representation of $S_{nk+1},$   for the pure Cohen-Macaulay set partition poset of $nk+1$ with block sizes congruent to 1 modulo $k$   \cite[Theorem 4.7]{CHR}, from which one deduces that
\begin{equation}\label{Pleth8f2} 
\sum_{n\equiv 1 {\rm\, mod\, }k} e_n {\rm\ and\ }\sum_{n\geq 0} (-1)^n \omega(\beta_{nk+1}) ; 
\end{equation}
when $k$ is \textbf{even}, and in particular when $\mathbf{k=2},$ and  $\beta_{nk+1} $ is as above.

\begin{equation}  \label{Pleth8g}
\sum_{n\geq 0} \eta_{2n+1} {\rm\ and\ }  \sum_{m\geq 0} (-1)^m \dfrac{\partial}{\partial p_1} \beta_{2m};
\end{equation}
here $\eta_{n}$ is the Frobenius characteristic of the $S_n$-action on the multilinear component of the free Jordan algebra on $n$ generators \cite[Proposition 3.5]{CHS}, while $\beta_{2m}$ is the homology representation of $S_{2m}$ for the set partition poset of ${2m}$ with even block sizes \cite[Theorem 4.7]{CHR}, from which one deduces that

\begin{equation} \label{Pleth8g2}
\sum_{n\geq 0} \omega(\eta_{2n+1}) {\rm\ and\ }  \sum_{m\geq 0} (-1)^m \dfrac{\partial}{\partial p_1} \omega(\beta_{2m}).
\end{equation}

\end{prop}

\begin{proof}  Using the property of plethysm that  $p_n[f]=f[p_n]$ for all symmetric functions $f,$ the first three pairs are easily checked to be plethystic inverses  by direct computation and the defining property of the M\"obius function in the divisor lattice.

Equation (\ref{Pleth8d}) is well known  as a consequence of the acyclicity of the (sign-twisted) Whitney homology for the partition lattice e.g.\cite[Theorem 1.8 and Remark 1.8.1]{Su0}.  The derivation we give here shows that it is  in fact equivalent to  Thrall's decomposition  of the regular representation, namely  $H[Lie]=(1-p_1)^{-1}.$ Applying (1) of Lemma \ref{pm} to the latter identity immediately gives (\ref{Pleth8d}).  Similarly, (\ref{Pleth8e}) is equivalent to our decomposition $E[Lie^{(2)}]=(1-p_1)^{-1}$ of the regular representation, see (the first equation in) Theorem \ref{PlInv-E},  as can be seen by invoking (2) of Lemma \ref{pm}. 

The inverse pair (\ref{Pleth8f}) is a result of \cite[Theorem 4.7 (a), p. 297]{CHR}.   It  also reflects the acyclicity of Whitney homology for the poset of partitions of a set of size $nk+1$ into blocks of size congruent to 1 modulo $k$  (see \cite[Example 3.7]{SuJer}). 
The pair (\ref{Pleth8f2}) follows by applying the involution $\omega$ to (\ref{Pleth8f}), since, when $k$ is even, all the symmetric functions involved have odd degree, and we have $\omega(f[g])
=\omega(f)[\omega[g]$ when $g$ is homogeneous of odd degree.  

Finally, the inverse pair (\ref{Pleth8g}) is a result of \cite[Proposition 3.5]{CHS}, with ~\eqref{Pleth8g2} following as in the preceding case.
\end{proof}

We conclude with two questions.

\begin{qn} Is there a formula describing, for each partition  $\lambda,$ the irreducible decomposition of $H_\lambda[Lie_n^{S}],$ the higher $Lie_n^{S}$-module for arbitrary subsets $S$ of primes?   This is open for all $S,$   including  the classical $Lie$ case (Thrall's problem) when $S=\emptyset.$   Note that Theorem~\ref{Foulkes} gives a nice combinatorial formula, in terms of standard Young tableaux and the major index, for the irreducible decomposition of $Lie_n^S$ itself.
\end{qn}

\begin{qn}  A theorem of Gessel and Reutenauer states that the higher Lie module $H_\lambda[Lie_n]$ is the fundamental quasi-symmetric function corresponding to the conjugacy class indexed by $\lambda$ \cite[Theorem 3.6]{GR}. Is there an analogous theory  corresponding to $H_\lambda[Lie_n^{(q)}]$ or $E_\lambda[Lie_n^{(q)}]$  when $q$ is prime? The case $q=2$ is of particular interest in view of the curious properties of $Lie_n^{(2)}.$
\end{qn}




\bibliographystyle{amsplain.bst}

\end{document}